\documentclass[12pt]{amsart}
\pdfoutput=1

\usepackage[utf8]{inputenc}
\usepackage[T1]{fontenc}
\usepackage{lmodern}
\usepackage{amssymb}
\usepackage[marginratio={1:1,1:1}]{geometry}
\usepackage{microtype}
\usepackage{hyperref}

\theoremstyle{plain}
\newtheorem{thm}{Theorem}[section]
\newtheorem{lem}[thm]{Lemma}
\newtheorem{cor}[thm]{Corollary}
\newtheorem{prop}[thm]{Proposition}

\theoremstyle{definition}
\newtheorem{defn}[thm]{Definition}
\newtheorem{q}[thm]{Question}
\newtheorem{ex}[thm]{Example}
\newtheorem{notn}[thm]{Notation}
\newtheorem{rem}[thm]{Remark}

\theoremstyle{remark}
\newtheorem{step}{Step}

\numberwithin{equation}{section}

\newcommand{\N}{\mathbb N}
\newcommand{\Z}{\mathbb Z}
\newcommand{\Q}{\mathbb Q}
\newcommand{\R}{\mathbb R}
\newcommand{\C}{\mathbb C}
\newcommand{\T}{\mathbb T}
\renewcommand{\O}{\Omega}

\newcommand{\CC}{\mathcal C}
\newcommand{\FF}{\mathcal F}
\newcommand{\OO}{\mathcal O}
\newcommand{\QQ}{\mathcal Q}
\newcommand{\BQ}{\overline\QQ}
\newcommand{\UU}{\mathcal U}
\newcommand{\VV}{\mathcal V}
\newcommand{\WW}{\mathcal W}

\newcommand{\rt}{\textup{rt}}
\newcommand{\lt}{\textup{lt}}
\newcommand{\axb}{{\alpha^{\textup{aff}}}}

\newcommand{\what}{\widehat}
\newcommand{\wilde}{\widetilde}
\newcommand{\ghat}{\what G}
\newcommand{\hhat}{\what H}
\newcommand{\ohat}{\what\O}
\newcommand{\nhat}{\what N}
\newcommand{\nshat}{\what{N^*}}

\newcommand{\under}{\backslash}
\newcommand{\iso}{\overset{\cong}{\longrightarrow}}
\newcommand{\midtext}[1]{\quad\text{#1}\quad}

\DeclareMathOperator{\aut}{Aut}
\DeclareMathOperator{\ad}{Ad}

\newcommand{\id}{\textup{id}}
\newcommand{\inv}{^{-1}}

\newcommand{\secref}[1]{Section~\ref{#1}}
\newcommand{\subsecref}[1]{Subsection~\ref{#1}}
\newcommand{\thmref}[1]{Theorem~\ref{#1}}
\newcommand{\corref}[1]{Corollary~\ref{#1}}
\newcommand{\lemref}[1]{Lemma~\ref{#1}}
\newcommand{\propref}[1]{Proposition~\ref{#1}}
\newcommand{\remref}[1]{Remark~\ref{#1}}
\newcommand{\qref}[1]{Question~\ref{#1}}
\newcommand{\exref}[1]{Example~\ref{#1}}
\newcommand{\appxref}[1]{Appendix~\ref{#1}}

\begin{document}

\title{Cuntz-Li algebras from $a$-adic numbers}

\author[Kaliszewski]{S.~Kaliszewski}
\address{School of Mathematical and Statistical Sciences
\\Arizona State University
\\Tempe, Arizona 85287}
\email{kaliszewski@asu.edu}
\author[Omland]{Tron Omland}
\address{Department of Mathematical Sciences
\\Norwegian University of Science and Technology (NTNU)
\\NO-7491 Trondheim, Norway}
\email{tron.omland@math.ntnu.no}
\author[Quigg]{John Quigg}
\address{School of Mathematical and Statistical Sciences
\\Arizona State University
\\Tempe, Arizona 85287}
\email{quigg@asu.edu}



\date{December 14, 2015}

\begin{abstract}
The $a$-adic numbers are those groups that arise as Hausdorff completions of noncyclic subgroups of the rational numbers $\Q$. We give a crossed product construction of (stabilized) Cuntz-Li algebras coming from the $a$-adic numbers and investigate the structure of the associated algebras. In particular, these algebras are in many cases Kirchberg algebras in the UCT class. Moreover, we prove an $a$-adic duality theorem, which links a Cuntz-Li algebra with a corresponding dynamical system on the real numbers.

The paper also contains an appendix where a nonabelian version of the ``subgroup of dual group theorem'' is given in the setting of coactions.
\end{abstract}

\maketitle

\section*{Introduction}

In \cite{cuntzq} Cuntz introduces the $C^*$-algebra $\QQ_{\N}$ associated with the $ax+b$-semigroup over the natural numbers, that is, $\Z\rtimes\N^{\times}$, where $\N^{\times}$ acts on $\Z$ by multiplication. It is defined as the universal $C^*$-algebra generated by isometries $\{s_n\}_{n\in\N^{\times}}$ and a unitary $u$ satisfying the relations
\[
s_ms_n=s_{mn},\quad s_nu=u^ns_n,
\midtext{and}
\sum_{k=0}^{n-1}u^ks_n^{\phantom{*}}s_n^*u^{-k}=1
\quad\text{for $m,n\in\N^{\times}$.}
\]
The algebra $\QQ_{\N}$ may be concretely realized on $\ell^2(\Z)$ equipped with the standard orthonormal basis $\{\delta_n\}_{n\in\Z}$ by
\[
s_m(\delta_n)=\delta_{mn}\midtext{and}u(\delta_n)=\delta_{n+1}.
\]
Furthermore, $\QQ_{\N}$ is shown to be simple and purely infinite
and can also be obtained as a semigroup crossed product
\[
C(\what{\Z})\rtimes(\Z\rtimes\N^{\times})
\]
for the natural $ax+b$-semigroup action of $\Z\rtimes\N^{\times}$ on the finite integral adeles $\what{\Z}=\prod_{p\text{ prime}}\Z_p$ (i.e., $\what{\Z}$ is the profinite completion of $\Z$). Its stabilization $\BQ_{\N}$ is isomorphic to the ordinary crossed product
\[
C_0(\mathcal{A}_f)\rtimes(\Q\rtimes\Q^{\times}_{+}),
\]
where $\Q^{\times}_{+}$ denotes the multiplicative group of positive rationals and $\mathcal{A}_f$ denotes the finite adeles,
i.e., the restricted product $\prod_{p\text{ prime}}'\Q_p=\prod_{p\text{ prime}} (\Q_p,\Z_p)$. The action of $\Q\rtimes\Q^{\times}_{+}$ on $\mathcal{A}_f$ is the natural $ax+b$-action. This crossed product is the minimal automorphic dilation of the semigroup crossed product above (see Laca \cite{lac:corner}).

Replacing $\N^{\times}$ with $\Z^{\times}$ gives rise to the $C^*$-algebra $\QQ_{\Z}$ of the ring $\Z$.
This approach is generalized to certain integral domains by Cuntz and Li \cite{CLintegral} 
and then to more general rings by Li \cite{Li-Ring}. 

In \cite{LarsenLi} Larsen and Li define the $2$-adic ring algebra of the integers $\QQ_2$, attached to the semigroup $\Z\rtimes \lvert 2\rangle$, where $\lvert 2\rangle=\{2^i:i\geq 0\}\subset\N^{\times}$ acts on $\Z$ by multiplication. It is the universal $C^*$-algebra generated by an isometry $s_2$ and a unitary $u$ satisfying the relations
\[
s_2u^k=u^{2k}s_2\midtext{and}s_2^{\phantom{*}}s_2^*+us_2^{\phantom{*}}s_2^*u^*=1.
\]
The algebra $\QQ_2$ shares many structural properties with $\QQ_{\N}$. It is simple, purely infinite and has a semigroup crossed product description. Its stabilization $\BQ_2$ is isomorphic to its minimal automorphic dilation, which is the crossed product
\[
C_0(\Q_2)\rtimes \bigl(\Z[\tfrac{1}{2}]\rtimes\langle 2\rangle\bigr).
\]
Here, $\Z[\tfrac{1}{2}]$ denotes the ring extension of $\Z$ by $\tfrac{1}{2}$, $\langle 2\rangle$ the subgroup of the positive rationals $\Q^{\times}_{+}$ generated by $2$ and the action of $\Z[\frac{1}{2}]\rtimes\langle 2\rangle$ on $\Q_2$ is the natural $ax+b$-action. 

Both $\mathcal{A}_f$ and $\Q_2$ are examples of groups of so-called $a$-adic numbers, defined by a doubly infinite sequence $a=(\dotsc,a_{-2},a_{-1},a_0,a_1,a_2,\dotsc)$ with $a_i\geq 2$ for all $i\in\Z$. For example, if $p$ is a prime number, the group $\Q_p$ of $p$-adic numbers is associated with the sequence $a$ given by $a_i=p$ for all $i$.

Our goal is to construct $C^*$-algebras associated with the $a$-adic numbers and show that these algebras provide a family of examples that under certain conditions share many structural properties
with $\QQ_2$, $\QQ_{\N}$, and also the ring $C^*$-algebras of Cuntz and Li. In particular, since $\N^\times$ is generated by the set of prime numbers, we can associate $\QQ_{\N}$ with the set of all primes, and $\QQ_2$ with the set consisting of the single prime $2$. In the same way, one can construct algebras $\QQ_P$ associated with any nonempty subset $P$ of the prime numbers, with similar generators and relations as described above. The stabilized algebras $\BQ_P$ provide one interesting class of $C^*$-algebras coming from the $a$-adic numbers.

Our approach is inspired by \cite{KLQcuntz}, that is, we begin with a crossed product and use the classical theory of $C^*$-dynamical systems to prove our results, instead of the generators and relations as in the papers of Cuntz, Li, and Larsen. 
Therefore, our construction only gives analogues of the stabilized algebras $\BQ_{\N}$ and $\BQ_2$.

Even though the $C^*$-algebras associated with $a$-adic numbers are closely related to the ring $C^*$-algebras of Cuntz and Li,
they are not a special case of these (except in the finite adeles case). Also, our approach does not fit in general into the framework of \cite{KLQcuntz}.

One of the main results in this paper is an $a$-adic duality theorem (Theorem~\ref{a-adic-duality}), which generalizes the $2$-adic duality theorem \cite[Theorem~7.5]{LarsenLi} and the analogous result of Cuntz \cite[Theorem~6.5]{cuntzq}. In the proof, we only apply crossed product techniques, and not the groupoid equivalence as in \cite{LarsenLi}. Another advantage with this strategy is that we in a natural way obtain a concrete bimodule for the Morita equivalence.

In the first section we describe the $a$-adic numbers $\O$ as the Hausdorff completion of a subgroup $N$ of $\Q$ and explain that this approach coincides with the classical one in Hewitt and Ross \cite{hr}. Then we go on and introduce the Cuntz-Li algebras associated with a sequence $a$, that is, coming from an $ax+b$-action of $N\rtimes H$ on $\O$ for a certain multiplicative group $H$ contained in $N$, and show that these algebras in many cases have several nice properties. 

The proof of \thmref{a-adic-duality} relies especially on two other results; a duality result for groups in \secref{dual-groups} describing the dual group $\nhat$ of $N$ for any noncyclic subgroup $N$ of $\Q$, and the ``subgroup of dual group theorem'', that we prove in a more general setting in \appxref{appendix}.

Finally, in \secref{isomorphisms} we characterize the $a$-adic numbers up to isomorphism, and give some isomorphism invariants for the associated Cuntz-Li algebras.

\vspace{1em}

This work was initiated during the second author's visit at Arizona State University, Spring 2012, and he would like to thank the other authors for their hospitality. A summary of the results of this paper is published in the conference proceedings \cite{faroes}.

The second author was partially supported by the Research Council of Norway.

\section{\texorpdfstring{The $a$-adic numbers}{The a-adic numbers}}

Let $a=(\dotsc,a_{-2},a_{-1},a_0,a_1,a_2,\dotsc)$ be a doubly infinite sequence of natural numbers with $a_i\geq 2$ for all $i\in\Z$. Let the sequence $a$ be arbitrary, but fixed.

We use Hewitt and Ross \cite[Section~10 and~25]{hr} as our main reference (see also \cite[12.3.35]{Palmer2}) and define the $a$-adic numbers $\O$ as the group of sequences
\[
\left\{ x=(x_i)\in\prod_{i=-\infty}^{\infty}\{0,1,\dotsc,a_i-1\} : x_i=0\text{ for }i<j\text{ for some }j\in\Z\right\}
\]
under addition with carry, that is, the sequences have a first nonzero entry and addition is defined inductively. Its topology is generated by the subgroups $\{\OO_j:j\in\Z\}$, where
\[
\OO_j=\{x\in\O:x_i=0\text{ for }i<j\}.
\]
This makes $\O$ a totally disconnected, locally compact Hausdorff abelian group. The group $\Delta$ of $a$-adic integers is defined as $\Delta=\OO_0$. It is a compact, open subgroup, and a maximal compact ring in $\O$ with product given by multiplication with carry. On the other hand, $\O$ itself is not necessarily a topological ring (see \thmref{herman-thm}).

Define the $a$-adic rationals $N$ as the additive subgroup of $\Q$ given by
\[
N=\left\{\frac{j}{a_{-1}\dotsm a_{-k}}:j\in\Z, k\geq 1\right\}.
\]
In fact, all noncyclic additive subgroups of $\Q$ containing $\Z$ are of this form (see \lemref{supernatural} below). There is an injective homomorphism
\[
\iota\colon N\hookrightarrow\O
\]
determined by
\[
\Bigl(\iota\Bigl(\frac{1}{a_{-1}\dotsm a_{-k}}\Bigr)\Bigr)_{-j}=\delta_{jk}.
\]
Moreover, $\iota(N)$ is the dense subgroup of $\O$ comprising the sequences with only finitely many nonzero entries. This map restricts to an injective ring homomorphism (denoted by the same symbol)
\[
\iota\colon\Z\hookrightarrow\Delta
\]
with dense range. Henceforth, we will suppress the $\iota$ and identify $N$ and $\Z$ with their image in $\O$ and $\Delta$, respectively.

Now let $\UU$ be the family of all subgroups of $N$ of the form $\tfrac{m}{n}\Z$, where $m$ and $n$ are natural numbers such that $m$ divides $a_0\dotsm a_j$ for some $j\geq 0$ and $n$ divides $a_{-1}\dotsm a_{-k}$ for some $k\geq 1$. Then $\UU$
\begin{enumerate}
\item is downward directed, that is, for all $U,V\in\UU$ there exists $W\in\UU$ such that $W\subset U\cap V$,
\item is separating, that is,
\[\bigcap_{U\in\UU} U=\{e\},\]
\item has finite quotients, that is, $\lvert U/V\rvert<\infty$ whenever $U,V\in\UU$ and $V\subset U$,
\end{enumerate}
and the same is also true for
\[
\VV=\{U\cap\Z:U\in\UU\}.
\]
In fact, both $\UU$ and $\VV$ are closed under intersections, since
\[
\frac{m}{n}\Z\cap\frac{m'}{n'}\Z=\frac{\operatorname{lcm}{(m,m')}}{\operatorname{gcd}{(n,n')}}\Z.
\]
It is a consequence of (i)-(iii) above that the collection of subgroups $\UU$ induces a locally compact Hausdorff topology on $N$.
Denote the Hausdorff completion of $N$ with respect to this topology by $\overline{N}$. Then
\[
\overline{N}\cong\varprojlim_{U\in\UU}N/U.
\]
Next, for $j\ge 0$ define
\[
U_j=\begin{cases}
\Z &\text{if } j=0,\\
a_0\cdots a_{j-1}\Z &\text{if } j\ge 1,
\end{cases}
\]
and set
\[
\WW=\{U_j:j\geq 0\}\subset\VV\subset\UU.
\]
Note that $\WW$ is also separating and closed under intersections. The closure of $U_j$ in $\O$ is $\OO_j$, so 
\[
\O/\OO_j\cong N/U_j\midtext{and}\Delta/\OO_j\cong \Z/U_j
\quad\text{for all $j\geq 0$}.
\]
Next, let
\[
\tau_j\colon\O\to N/U_j
\]
denote the quotient map for $j\geq 0$, and identify $\tau_j(x)$ with the truncated sequence $x^{(j-1)}$, where $x^{(j)}$ is defined for all $j\in\Z$ by
\[
(x^{(j)})_i=\begin{cases}x_i & \text{ for }i\leq j, \\ 0 & \text{ for }i> j. \end{cases}
\]
We find it convenient to use the standard construction of the inverse limit of the system $\{N/U_j,\pmod{a_j}\}$:
\[
\varprojlim_{j\geq 0}N/U_j = \left\{ x=(x_i)\in\prod_{i=0}^{\infty}N/U_i : x_i=x_{i+1}\pmod{a_i}\right\},
\]
and then the product $\tau\colon\O\to\varprojlim_{j\geq 0}N/U_j$ of the truncation maps $\tau_j$, given by
\[
\tau(x)=(\tau_0(x),\tau_1(x),\tau_2(x),\dotsc), 
\]
is an isomorphism. Furthermore, we note that $\WW$ is cofinal in $\UU$. Indeed, for all $U=\tfrac{m}{n}\Z\in\UU$, if we choose $j\geq 0$ such that $m$ divides $a_0\dotsm a_j$ then we have $\WW\ni U_{j+1}\subset U$. Therefore, 
\[
\O \cong \varprojlim_{j\geq 0}N/U_j \cong \varprojlim_{U\in\UU}N/U \cong \overline{N},
\]
and similarly 
\[
\Delta \cong \varprojlim_{j\geq 0}\Z/U_j \cong \varprojlim_{V\in\VV}\Z/V \cong \overline{\Z}.
\]
In particular, $\Delta$ is a profinite group. In fact, every profinite group coming from a completion of $\Z$ occurs this way (see \lemref{supernatural} below).
The following should serve as motivation for our definition of $\UU$.
\begin{lem}\label{open-subgroups}
Every open subgroup of $\O$ is of the form
\[
\overline{\bigcup_{U\in\CC}U}
\]
for some increasing chain $\CC$ in $\UU$.

In particular, every compact open subgroup of $\O$ is of the form $\overline{U}$ for some $U\in\UU$.
\end{lem}
\begin{proof}
Let $U,V\in\UU$ and suppose there are $x\in\overline{U}\setminus\overline{V}$ and $y\in\overline{V}\setminus\overline{U}$. Then $x+y\notin\overline{U}\cup\overline{V}$, so $\overline{U}\cup\overline{V}=\overline{U\cup V}$ is a subgroup of $\O$ if and only if $U\subset V$ or $V\subset U$. Moreover, $\bigcup_{U\in\CC}\overline{U}$ is open in $\O$, hence also closed whenever it is a subgroup \cite[5.5]{hr}, so it equals $\overline{\bigcup_{U\in\CC}U}$.

Note that $\UU$ is closed under intersections and $\overline{U}\cap\overline{V}=\overline{U\cap V}$ by construction. Indeed, since $\overline{U}\cap\overline{V}$ is open in $\O$ and $N$ is dense in $\O$, $(\overline{U}\cap\overline{V})\cap N=U\cap V$ is dense in $\overline{U}\cap\overline{V}$.
\end{proof}
\begin{notn}
Whenever any confusion is possible, we write $\O_a$, $\Delta_a$, $N_a$, etc.\ for the structures associated with the sequence $a$. If $a$ and $b$ are two sequences such that $\mathcal{U}_a=\mathcal{U}_b$, we write $a\sim b$. In this case also $N_a=N_b$. It is not hard to verify that $a\sim b$ if and only if there is an isomorphism $\O_a\to\O_b$ restricting to an isomorphism $\Delta_a\to\Delta_b$. The groups $\O_a$ and $\O_b$ can nevertheless be isomorphic even if $a\not\sim b$ (see \exref{three-at-negative} and \corref{omega-iso} below).
\end{notn}
\begin{ex}\label{single-prime}
Let $p$ be a prime and assume $a=(\dotsc,p,p,p,\dotsc)$. Then $\O\cong\Q_p$ and $\Delta\cong\Z_p$, i.e., the usual $p$-adic numbers and $p$-adic integers.
\end{ex}
\begin{ex}\label{all-primes}
Let $a=(\dotsc,4,3,2,3,4,\dotsc)$, i.e., $a_i=a_{-i}=i+2$ for $i\geq 0$. Then $\O\cong\mathcal{A}_f$ and $\Delta\cong\what{\Z}$, because every prime occurs infinitely often among both the positive and the negative tail of the sequence $a$.
\end{ex}
\begin{ex}\label{three-at-zero}
Let $a_i=2$ for $i\neq 0$ and $a_0=3$, so that
\[
N=\Z[\tfrac{1}{2}]\midtext{ and }\UU=\{2^i\Z,2^i3\Z:i\in\Z\}.
\]
Then $\O$ contains torsion elements. Indeed, let
\[
x=(\dotsc,0,1,1,0,1,0,1,\dotsc),
\midtext{so that} 2x=(\dotsc,0,2,0,1,0,1,0,\dotsc),
\]
where the first nonzero entry is $x_0$. Then $3x=0$ and $\{0,x,2x\}$ forms a subgroup of $\O$ isomorphic with $\Z/3\Z$. Hence, $\O\not\cong\Q_2$ since $\Q_2$ is torsion-free.
\end{ex}
\begin{ex}\label{three-at-negative}
Assume $a$ is as in the previous example and let $b$ be given by $b_i=a_{i+1}$, that is, $b_i=2$ for $i\neq -1$ and $b_{-1}=3$. Then
\[
N_b=\tfrac{1}{3}\Z[\tfrac{1}{2}]\midtext{ and }\UU_b=\{2^i\Z,2^i\tfrac{1}{3}\Z:i\in\Z\}.
\]
We have $\O_a\cong\O_b$ (see \lemref{supernatural} and \propref{omega-isomorphism}), but $a\not\sim b$ since $\Delta_a\not\cong\Delta_b$.
\end{ex}
\begin{rem}
One should think of the entries from the negative tail of $a$ as determining a subgroup $N$ of $\Q$, and the positive tail as determining a topology that gives rise to a completion of $N$. The position of $a_0$ does not have any impact on the additive structure (up to isomorphism), but matters with respect to the multiplicative structure.

Moreover, instead of thinking about the $a$-adic numbers as coming from a sequence of numbers, one may consider them as coming from two chains of ideals. Indeed, every downward directed chain of ideals of $\Z$ (ordered by inclusion) has the form
\[
\Z\supset a_0\Z\supset a_0a_1\Z\supset a_0a_1a_2\Z\supset\dotsm
\]
for some sequence $(a_0,a_1,a_2,\dotsc)$ of integers $a_i\geq 2$. In this way, it may be possible to generalize the concept of $a$-adic numbers to other rings, for example to rings of integers of global fields.
\end{rem}

\section{\texorpdfstring{The $a$-adic algebras}{The a-adic algebras}}

We now want to define a multiplicative action on $\O$, of some suitable subset of $N$, that is compatible with the natural multiplicative action of $\Z$ on $\O$. Let $S$ consist of all $s\in\Q^{\times}_{+}$ such that the map $\UU\to\UU$ given by $U\mapsto sU$ is well-defined and bijective.

Clearly, the map $U\mapsto sU$ is injective if it is well-defined and it is surjective if the map $U\mapsto s\inv U$ is well-defined. In other words, $S$ consists of all $s\in\Q^{\times}_{+}$ such that the maps $\UU\to\UU$ given by $U\mapsto sU$ and $U\mapsto s\inv U$ are both well-defined.
\begin{lem}
If $s_1,s_2\in\N^{\times}$, $s=s_1s_2$, and the map $\UU\to\UU$ given by $U\mapsto sU$ is well-defined and bijective, then the maps $U\mapsto s_1U$ and $U\mapsto s_2U$ are also well-defined and bijective.
\end{lem}
\begin{proof}
First, we pick $t\in\N^{\times}$. Then $t\frac{m}{n}\Z\in\UU$ whenever $\frac{m}{n}\Z\in\UU$ if and only if $tm'\Z\in\UU$ whenever $m'\Z\in\UU$. One direction is obvious, so assume $tm'\Z\in\UU$ for all $m'\Z\in\UU$ and pick $\frac{m}{n}\Z\in\UU$ (with $m$ and $n$ coprime). Since $m\Z\in\UU$ as well, $tm\Z\in\UU$, so $tm$ divides $a_0\dotsm a_j$ for some $j\geq 0$, and thus the numerator of $\frac{tm}{n}$ (after simplifying) also divides $a_0\dotsm a_j$. And since $n$ divides $a_{-1}\dotsm a_{-k}$ for some $k\geq 1$, the denominator of $\frac{tm}{n}$ (after simplifying) also divides $a_{-1}\dotsm a_{-k}$.

Similarly, $t\inv\frac{m}{n}\Z\in\UU$ whenever $\frac{m}{n}\Z\in\UU$ if and only if $\frac{1}{tn'}\Z\in\UU$ whenever $\frac{1}{n'}\Z\in\UU$.

Therefore, to show that the map $U\mapsto s_iU$ is well-defined and bijective for $i=1,2$, it is enough to show that $s_im\Z,\frac{1}{s_in}\Z\in\UU$ whenever $m\Z,\frac{1}{n}\Z\in\UU$. But this follows immediately since $s_im$ divides $sm$, which again divides $a_0\dotsm a_j$ for some $j\geq 0$, and $s_in$ divides $sn$, which divides $a_{-1}\dotsm a_{-k}$ for some $k\geq 1$.
\end{proof}
\begin{cor}
The set $S$ is a subgroup of $\Q^{\times}_{+}$ and may also be described in the following way.
Define a set of prime numbers by 
\[
P=\{p\text{ prime}:p\text{ divides }a_k\text{ for infinitely many }k<0\text{ and infinitely many }k\geq 0\}.
\]
Then $S$ coincides with the subgroup $\langle P\rangle$ of $\Q^\times_{+}$ generated by $P$.
\end{cor}
\begin{proof}
The first statement is obvious from the previous argumentation, since if $s$ belongs to $S$, then $s^i$ belongs to $S$ for all $i\in\Z$.

If $s\in S\cap\N^{\times}$, it follows from the previous lemma that all its prime factors must be in $S$. Hence, $p^i$ belongs to $S$ for all $i\in\Z$ whenever $p$ is a prime factor of some $s\in S\cap\N^{\times}$.
\end{proof}
Therefore, $S$ consists of elements $s\in N$ such that the multiplicative actions of both $s$ and $s\inv$ on $N$ are well-defined and continuous with respect to $\UU$. It is well-defined since all $q\in N$ belongs to some $U\in\UU$. If $q+U$ is a basic open set in $N$, then its inverse image under multiplication by $s$, $s\inv (q+U)=s\inv q+s\inv U$, is also open in $N$ as $s\inv U\in\UU$. By letting $S$ be discrete, it follows that the action is continuous.

We will not always be interested in the action of the whole group $S$ on $N$, but rather a subgroup of $S$. So henceforth, let $H$ denote any subgroup of $S$. Furthermore, let $G$ be the semidirect product of $N$ by $H$, i.e., $G=N\rtimes H$, where $H$ acts on $N$ by multiplication. This means that there is a well-defined $ax+b$-action of $G$ on $N$ given by
\[
(r,h)\cdot q=r+hq\quad\text{ for $q,r\in N$ and $h\in H$.}
\]
The automorphisms of $N$ given by this action are continuous with respect to $\UU$, and therefore the action can be extended to $\O$, by uniform continuity.
\begin{ex}
To see why the primes in $P$ must divide infinitely many terms of both the positive and negative tail of the sequence, consider the following example. Let $a_i=2$ for $i<1$ and $a_i=3$ for $i\geq 1$ and let $x$ be defined by $x_i=0$ for $i<0$ and $x_i=1$ for $i\geq 0$. Then
\[
\tfrac{1}{2}\cdot (2\cdot x)=\tfrac{1}{2}\cdot 0=0\neq x=1\cdot x=(\tfrac{1}{2}\cdot 2)\cdot x.
\]
Here, $N=\Z[\tfrac{1}{2}]$ is a ring, but the problem is that multiplication by $\tfrac{1}{2}$ is not continuous on $N$ with respect to $\UU$. In particular, $x^{(j+1)}=3x^{(j)}$, so $\frac{1}{2}(x^{(j+1)}-x^{(j)})=x^{(j)}\not\to 0$, hence $(\frac{1}{2}x^{(j)})$ is not Cauchy. Note that $P=\varnothing$ in this case.
\end{ex}
\begin{prop}\label{action-properties}
Assume $P\neq\varnothing$ and let $H$ be a nontrivial subgroup of $S$.
Then the action of $G=N\rtimes H$ on $\O$ is minimal, locally contractive, and topologically free.
\end{prop}
\begin{proof}
For the minimality, just observe that for any $x\in\O$ the orbit $\{g\cdot x:g\in G\}$ is dense because it contains $x+N$, which is dense in $x+\O=\O$.

Next, note that for all $U\in\UU$ there is an $s\in H$ such that $sU\subsetneq U$. In fact, since $sU\in\UU$ for all $s\in H$, any $s\in H$ with $s>1$ will do the job. Therefore, for all open $q+\overline{U}\subset\O$, $q\in N$, pick $s\in H$ such that $s\cdot\overline{U}\subsetneq\overline{U}$ and put $r=q-sq\in N$. Then
\[
(r,s)\cdot (q+\overline{U})=s\cdot(q+\overline{U})+r=sq+s\cdot\overline{U}+r=q+s\cdot\overline{U}\subsetneq q+\overline{U},
\]
which means that the action is locally contractible.

Finally, suppose that $\{x\in\O:g\cdot x =x\}$ has nonempty interior for some $g=(r,s)$. Since $N$ is dense in $\O$, this implies that there is some open set $U$ in $N$ such that $g\cdot x=x$ for all $x\in U$. Assume that two distinct elements $x,y\in U$ both are fixed by $g=(r,s)$, that is, $g\cdot x=sx+r=x$ and $g\cdot y=sy+r=y$. Then $x-y=s(x-y)$, so $s=1$ since $x\neq y$, and hence $r=0$, which means that $g=(r,s)=(0,1)=e$.
\end{proof}
\begin{defn}
Suppose $P\neq\varnothing$, that is, $S\neq\{1\}$.
If $H$ is a nontrivial subgroup of $S$, we define the $C^*$-algebra $\BQ=\BQ(a,H)$ by
\[
\BQ = C_0(\O)\rtimes_{\axb} G,
\]
where
\[
\axb_{(n,h)}(f)(x)=f(h\inv \cdot (x-n)).
\]
\end{defn}
\begin{notn}
The bar-notation on $\BQ$ is used so that it agrees with the notation for stabilized Cuntz-Li algebras in \cite{cuntzq} and \cite{LarsenLi}.
\end{notn}
\begin{thm}
The $C^*$-algebra $\BQ$ is simple and purely infinite.
\end{thm}
\begin{proof}
This is a direct consequence of \propref{action-properties}. Indeed, $\BQ$ is simple since the action $\axb$ is minimal and topologically free \cite{as}, and $\BQ$ is purely infinite since $\axb$ is locally contractive \cite{LSboundary}.
\end{proof}
\begin{cor}
The $C^*$-algebra $\BQ$ is a nonunital Kirchberg algebra in the UCT class.
\end{cor}
\begin{proof}
The algebra is clearly nonunital and separable. It is also nuclear since the dynamical system consists of an amenable group acting on a commutative $C^*$-algebra. Finally, $\BQ$ may be identified with the $C^*$-algebra of the transformation groupoid attached to $(\O,G)$. This groupoid is amenable and hence the associated $C^*$-algebra belongs to the UCT class \cite{Tu}. 
\end{proof}
\begin{rem}
We should note the difference between the setup described here and the one in \cite{KLQcuntz}. Here, we do not assume that the family of subgroups $\UU$ is generated by the action of $H$ on some subgroup $M$ of $N$. That is, we do not assume that $\UU=\UU_M=\{h\cdot M:h\in H\}$ for any $M$, only that $\UU_M\subset\UU$. In light of this, we were lead to convince ourselves that the results in \cite[Section~3]{KLQcuntz} also hold under the following slightly weaker conditions on $\UU$. Let $G=N\rtimes H$ be a discrete semidirect product group with normal subgroup $N\neq G$ and quotient group $H$. Let $\UU$ be a family of normal subgroups of $N$ satisfying the following conditions:
\begin{enumerate}
\item $\UU$ is downward directed, separating, and has finite quotients.
\item For all $U\in\UU$ and $h\in H$, $h\cdot U\in\UU$.
\item For all $V\in\UU$, the family \[\UU_V=\{h\cdot V:h\in H\}\] is downward directed and separating.
\item For all $h\in H$ and $U\in\UU$, $U$ is not fixed pointwise by the action of $h$.
\end{enumerate}
In particular, the conditions do not depend on any specific subgroup $M$ of $N$.
\end{rem}
\begin{ex}\label{prime-generated}
If $a=(\dotsc,2,2,2,\dotsc)$ and $H=S=\langle 2\rangle$, then $\BQ$ is the algebra $\BQ_2$ of Larsen and Li \cite{LarsenLi}. More generally, if $p$ is a prime, $a=(\dotsc,p,p,p,\dotsc)$ and $H=S=\langle p\rangle$, we are in the setting of \exref{single-prime} and get algebras similar to $\BQ_2$.

If $a=(\dotsc,4,3,2,3,4,\dotsc)$ and $H=S=\Q^{\times}_{+}$, then we are in the setting of \exref{all-primes}. In this case $\BQ$ is the algebra $\BQ_{\N}$ of Cuntz \cite{cuntzq}.

Both these algebras are special cases of the most well-behaved situation, namely where $H=S$ and $a_i\in H\cap\N^{\times}$ for all $i\in\Z$. The algebras arising this way are completely determined by the set (finite or infinite) of primes $P$, and are precisely the kind of algebras that fit into the framework of \cite{KLQcuntz}. The cases described above are the two extremes, where $P$ consists of either one single prime or all primes.
\end{ex}
\begin{rem}
If $a\sim b$, then $S_a=S_b$ and $\BQ(a,H)=\BQ(b,H)$ for all nontrivial $H\subset S_a=S_b$.
\end{rem}
\begin{ex}\label{singly-generated}
If $a=(\dotsc,2,2,2,\dotsc)$ and $b=(\dotsc,4,4,4,\dotsc)$, then $a\sim b$. Hence, for all nontrivial $H\subset S=\langle 2\rangle$ we have $\BQ(a,H)=\BQ(b,H)$. However, if $H=\langle 4\rangle$, then $\BQ(a,S)\not\cong\BQ(a,H)$, as remarked after \qref{k-theory}.
\end{ex}
In light of this example, it could also be interesting to investigate the $ax+b$-action on $\O$ of other subgroups $G'$ of $N\rtimes S$. It follows from the proof of \propref{action-properties} that the action of $G'$ on $\O$ is minimal, locally contractive and topologically free if and only if $G'=M\rtimes H$, where $M\subset N$ is dense in $\O$ and $H\subset S$ is nontrivial.
\begin{lem}\label{dense}
A proper subgroup $M$ of $N$ is dense in $\O$ if and only if $M=qN$ for some $q\geq 2$ such that $q$ and $a_i$ are relatively prime for all $i\in\Z$.
\end{lem}
\begin{proof}
First note that $\overline{M}\subsetneq\overline{N}=\O$ if and only if there exists a subfamily $\UU'\subset\UU$ so that
\[
M\subset\bigcup_{U\in\UU'}U\subsetneq N.
\]
This holds if and only if $M$ is contained in a subgroup of $N$ of the form
\[
\left\{\frac{L\cdot j}{a_{-1}'\dotsm a_{-k}'}:j\in\Z, k\geq 1\right\}\subsetneq N,
\]
where $a_{-1}',a_{-2}',\dotsc$ is a sequence of natural numbers such that for all $i\geq 1$ there is $k\geq 1$ such that
$a_{-1}'a_{-2}'\dotsc a_{-i}'$ divides $a_{-1}a_{-2}\dotsc a_{-k}$, and where $L$ divides $a_0\dotsm a_j$ for some $j\geq 0$, but $L$ does not divide $a_{-1}a_{-2}\dotsc a_{-k}$ for any $k\geq 1$.

If $M\subsetneq N$ and $M$ is not contained in any subgroup of $N$ of this form, then the only possibility is that $M=qN$
for some $q\geq 2$ such that $q$ and $a_i$ are relatively prime for all $i\in\Z$.
\end{proof}
Of course, $M$ and $N$ are isomorphic if the condition of \lemref{dense} is satisfied (see \remref{q-isomorphism}).
\begin{rem}\label{automorphism}
Let $p$ be a prime and let $U,V\in\UU$. Then the map $\O\to\O$ given by $x\mapsto px$ is continuous and open. Thus, $p\overline{U}=\overline{pU}=\overline{U}$ if $pU\notin\UU$. If $p\overline{U}=\overline{V}$ and $pU\in\UU$, then $p\inv V\in\UU$.
Set
\[
Q=\{p\text{ prime}:p\text{ does not divide any }a_i\}.
\]
Then multiplication by a prime $p$ is an automorphism of $\O$ if and only if $p\in P\cup Q$. Indeed, if $p\in Q$, then $p\overline{U}=\overline{U}$ for all $U\in\UU$ (see also \remref{inverses}).
\end{rem}
\begin{prop}
Suppose $M$ is a subgroup of $N$ that is dense in $\O$, and let $H$ be a nontrivial subgroup of $S$. As usual, let $G=N\rtimes H$ and set $G'=M\rtimes H$. Then \begin{equation}\label{other-subgroup} C_0(\O)\rtimes_{\axb} G\cong C_0(\O)\rtimes_{\axb} G'.
\end{equation}
\end{prop}
\begin{proof}
Assume that $M=qN$, where $q$ and $a_i$ are relatively prime for all $i$. Then an isomorphism is determined by the map $\varphi\colon C_c(G',C_0(\O))\to C_c(G,C_0(\O))$ given by
\[
\varphi(f)(n,h)(x)=f(qn,h)(qx).
\qedhere
\]
\end{proof}
\begin{rem}
We complete this discussion by considering the $ax+b$-action on $\O$ of potentially larger groups than $N\rtimes S$. By \remref{automorphism}, $\tfrac{1}{p}\in\O$ when $p\in Q$, and it is possible to embed the subgroup
\begin{equation}\label{NQ}
N_Q=\{\tfrac{n}{q}\mid n\in N,q\in\langle Q\rangle\}\subset\Q
\end{equation}
in $\O$, where $\langle Q\rangle$ denotes the multiplicative subgroup of $\Q^{\times}_{+}$ generated by $Q$. The largest subgroup of $\Q\rtimes\Q^{\times}_{+}$ that can act on $\O$ through an $ax+b$-action is $N_Q\rtimes\langle P\cup Q\rangle$.

For example, $\Q$ itself can be embedded into $\Q_p$ for all primes $p$, but it is not hard to see that \lemref{closed-range} of the next section fails in this case since the image of the diagonal map $\Q\to\R\times\Q_p$ is not closed. In fact, the only groups $N\subset M\subset N_Q$ that give rise to the duality theorem (\thmref{duality}) are of the form $M=\tfrac{1}{q}N$ for $q\in\langle Q\rangle$. Moreover, $S$ is the largest subgroup of $\langle P\cup Q\rangle$ that acts on $M$, and of course \eqref{other-subgroup} also holds for all $H\subset S$ in this case.

It is possible to adjust the way we apply \thmref{duality} so that we also can consider the action of larger subgroups. For this to work we have to ``compensate'' in the duality theorem, as we explain more precisely in \remref{compensate}.

Finally, we remark that one may also involve the roots of unity of $\Q^{\times}$ in the multiplicative action (as in \cite{CLintegral}), that is, replace $H$ with $\{\pm h:h\in H\}=\{\pm 1\}\times H$. The associated algebras will then be of the form $\BQ\rtimes\Z/2\Z$. However, we restrict to the action of the torsion-free part of $\Q^{\times}$ in this paper.
\end{rem}

\section{\texorpdfstring{A duality theorem for $a$-adic groups}{A duality theorem for a-adic groups}}\label{dual-groups}

In this section we first give the duality result (\thmref{duality}) for the group of $a$-adic numbers following the approach of \cite{hr}, after which we outline a slightly modified approach that has some advantage.

For any $a$, let $a^*$ be the sequence given by $a^*_i=a_{-i}$. In particular, $(a^*)^*=a$. We now fix $a$ and write $\O$ and $\O^*$ for the $a$-adic and $a^*$-adic numbers, respectively. Let $x\in\O$ and $y\in\O^*$ and for $j\geq 1$ put
\[
z_j=e^{2\pi ix^{(j)}y^{(j)}/a_0},
\]
where the truncated sequences $x^{(j)}$ and $y^{(j)}$ are treated as their corresponding rational numbers in $N$. Then $z_j$ will eventually be constant. Indeed, note that
\[
x^{(j)}=\frac{x_{-k}}{a_{-1}\dotsm a_{-k}}+\dotsm+\frac{x_{-1}}{a_{-1}}+x_0+a_0x_1+\dotsm+a_0\dotsm a_{j-1}x_j
\]
for some $k>0$. Therefore, $a_0a_{-1}\dotsm a_{-j+1}x^{(j)}\in a_0\Z$ for all $j>k$. Similarly, there is an $l>0$ such that $a_0a_1\dotsm a_{j-1}y^{(j)}\in a_0\Z$ for all $j>l$. Hence, for all $j>\max{\{k,l\}}$
\[
x^{(j+1)}y^{(j+1)}-x^{(j)}y^{(j)}=0 \pmod{a_0\Z},
\]
i.e., $z_{j+1}=z_j$. We now define the pairing $\O\times\O^*\to\T$ by
\[
\langle x,y\rangle_{\O}=\lim_{j\to\infty}e^{2\pi ix^{(j)}y^{(j)}/a_0}.
\]
The pairing is a continuous homomorphism in each variable separately and gives an isomorphism $\O^*\to\ohat$. Indeed, to see that our map coincides with the one in \cite[25.1]{hr}, note the following. Suppose $k$ and $l$ are the largest numbers such that $x_i=0$ for $i<-k$ and $y_j=0$ for $j<-l$. Then
\[
\langle x,y\rangle_{\O}=\begin{cases}e^{2\pi i x^{(l)}y^{(k)}/a_0}&\text{ if }k+l\geq 0,\\1&\text{ if }k+l<0.\end{cases}
\]
From this it should also be clear that if $x\in N$ and $y\in N^*$, then
\[
\langle x,y\rangle_{\O}=e^{2\pi ixy/a_0}.
\]
Furthermore, for each $j$, the pairing of $\O$ and $\O^*$ restricts to an isomorphism $\OO_{-j+1}^*\to(\OO_j)^\perp$. This is also explained in \cite[25.1]{hr}. Hence, for all $j\in\Z$
\[
\O^*/\OO_{-j+1}^*\cong\ohat/(\OO_j)^\perp\cong\what{\OO_j}.
\]
\begin{lem}\label{closed-range}
The injection $\iota\colon N\to\R\times\O$ given by $q\mapsto (q,q)$ has discrete range, and $N$ may be considered as a closed subgroup of $\R\times\O$.
\end{lem}
\begin{proof}
Since $(-1,1)$ is open in $\R$ and $\Delta$ is open in $\O$, we have $U=(-1,1)\times\Delta$ is open in $\R\times\O$. Moreover, $\iota(N)\cap U=\{0\}$ as $N\cap\Delta=\Z$. Therefore, $\{0\}$ is an isolated point in $\iota(N)$, hence the image of $\iota$ is discrete by \cite[5.8]{hr} which means it is closed in $\R\times\O$ by \cite[5.10]{hr}.
\end{proof}
Similarly, $N^*$ may be considered as a closed subgroup of $\R\times\O^*$. By applying the facts about the pairing of $\O$ and $\O^*$ stated above, the pairing of $\R\times\O$ and $\R\times\O^*$ given by
\[
\langle (u,x),(v,y)\rangle=e^{-2\pi iuv/a_0}\lim_{j\to\infty}e^{2\pi ix^{(j)}y^{(j)}/a_0}=\langle u,v\rangle_{\R}\langle x,y\rangle_{\O}
\]
defines an isomorphism $\varphi\colon\R\times\O^*\to\what{\R\times\O}$.
\begin{lem}
The map $\varphi$ restricts to an isomorphism $\iota(N^*)\to \iota(N)^\perp$.
\end{lem}
\begin{proof}
First, for $r\in N^*$ we have
\[
\langle (q,q),(r,r)\rangle=e^{-2\pi iqr/a_0}e^{2\pi iqr/a_0}=1
\quad\text{for all $q\in N$,}
\]
so $\iota(r)\in \iota(N)^\perp$.

On the other hand, let $(v,y)\in \iota(N)^\perp$, so that
\[
\langle (q,q),(v,y)\rangle=1\quad\text{for all $q\in N$.}
\]
Then
\[
\frac{1}{a_0}q(y^{(j)}-v)=0\pmod{\Z}\quad\text{for all $q\in N$ and sufficiently large $j$.}
\]
In particular, this must hold for $q=\frac{1}{a_{-1}\dotsm a_{-k}}$ for all $k\geq 1$, so
\[
y^{(j)}-v=0\pmod{a_0\dotsm a_{-k}\Z}\quad\text{for all $k\geq 1$ and sufficiently large $j$.}
\]
Since $y^{(j)}$ is rational for all $j$, the real number $v$ must also be rational. Moreover, $y^{(j)}=v$ for sufficiently large $j$, so the sequence $\{y^{(j)}\}_j$ must eventually be constant. That is, $y=y^{(j)}$ for large $j$, hence $y\in N^*$ and $y=v$, so
\[
(v,y)=(y,y)\in \iota(N^*).
\qedhere
\]
\end{proof}
Thus, we get the following theorem.
\begin{thm}\label{duality}
There are isomorphisms
\[
(\R\times\O^*)/N^*\overset{\cong}{\longrightarrow}\what{\R\times\O}/N^\perp\overset{\cong}{\longrightarrow}\nhat,
\]
and the isomorphism $\omega\colon(\R\times\O^*)/N^*\to\nhat$ is given by
\[
\omega\bigl((v,y)+N^*\bigr)(q)=\langle (q,q),(v,y)\rangle
\quad\text{for $(v,y)\in \R\times\O^*$ and $q\in N$.}
\]
\end{thm}
In particular, by this result one can describe the dual group of all noncyclic subgroups of $\Q$. These groups are compact and connected and are called the solenoids. For example, $(\R\times\mathcal{A}_f)/\Q\cong\what{\Q}$.
\begin{rem}\label{dense-dual}
Let $M$ be a subgroup of $N$ that is dense in $\O$, so that by \lemref{dense}, $M$ is of the form $qN$ for some $q\geq 2$. Then $M^*=qN^*\cong N^*$ is a subgroup of $N^*$ that is dense in $\O^*$. By replacing $M$ with $N$, the above argument gives $(\R\times\O^*)/M^*\cong \what M$.
\end{rem}

\subsection{\texorpdfstring{Modified duality results for $a$-adic groups}{Modified duality results for a-adic groups}}\label{modified-dual}

For any $a$ and integer $n$, let $a^{(n)}$ denote the sequence given by $a^{(n)}_i=a_{-i+n}$. We have $(a^{(n)})^{(n)}_i=a^{(n)}_{-i+n}=a_{-(-i+n)+n}=a_i$, so $(a^{(n)})^{(n)}=a$. In particular, we set $a^*=a^{(0)}$ and $a^\#=a^{(-1)}$.

Now we fix some sequence $a$, and our goal is to explain why $a^\#$ might be a better choice than $a^*$ for our purposes. First of all, note that if $\O^{(n)}$ and $\O^{(m)}$ are the $a$-adic numbers coming from $a^{(n)}$ and $a^{(m)}$, respectively, then $\O^{(n)}$ and $\O^{(m)}$ are isomorphic by \propref{omega-isomorphism} and the succeeding remark.

Let $x\in\O$ and $y\in\O^\#$ and put
\[
z_j=e^{2\pi ix^{(j)}y^{(j)}},
\]
where $x^{(j)}$ and $y^{(j)}$ are the truncations of $x$ and $y$. Then $z_j$ will eventually be constant. Indeed, note that
\[
x^{(j)}=\frac{x_{-k}}{a_{-1}\dotsm a_{-k}}+\dotsm+\frac{x_{-1}}{a_{-1}}+x_0+a_0x_1+\dotsm+a_0\dotsm a_{j-1}x_j
\]
for some $k>0$. Therefore, $a_{-1}\dotsm a_{-j+1}x^{(j)}\in\Z$ for all $j>k$. Similarly, there is an $l>0$ such that $a_0a_1\dotsm a_{j-1}y^{(j)}\in\Z$ for all $j>l$. Hence, for all $j>\max{\{k,l\}}$
\[
x^{(j+1)}y^{(j+1)}-x^{(j)}y^{(j)}=0 \pmod{\Z},
\]
i.e., $z_{j+1}=z_j$. We now define the pairing $\O\times\O^\#\to\T$ by
\[
\langle x,y\rangle_{\O}^{\#}=\lim_{j\to\infty}e^{2\pi ix^{(j)}y^{(j)}}.
\]
The pairing is also in this case a continuous homomorphism in each variable separately and gives an isomorphism $\O^\#\to\ohat$ by a similar method as in \cite[25.1]{hr}. From this it should also be clear that if $x\in N$ and $y\in N^\#$, then
\[
\langle x,y\rangle_{\O}^{\#}=e^{2\pi ixy}.
\]
Moreover, for all $j$
\[
\O^\#/\OO_{-j}^\#\cong\ohat/(\OO_j)^\perp\cong\what{\OO_j},
\]
and there is an isomorphism $\omega\colon(\R\times\O^\#)/N^\#\to\nhat$ given by
\[
\omega\bigl((v,y)+N^\#\bigr)(q)=\langle (q,q),(v,y)\rangle)
\quad\text{for $(v,y)\in \R\times\O^\#$ and $q\in N$.}
\]
\begin{rem}
Let $a$ and $b$ be two sequences. Then $a\sim b$ if and only $a^\#\sim b^\#$. On the other hand, $a\sim b$ does not imply $a^*\sim b^*$ in general (see \remref{a-sharp} and the preceding comments).
\end{rem}
\begin{rem}
Let $x\in\O$ and $y\in\O^{(n)}$ and put
\[
z_j=\begin{cases}e^{2\pi ia_{-1}\dotsm a_{n+1}x^{(j)}y^{(j)}} & n\leq -2,\\
e^{2\pi ix^{(j)}y^{(j)}} & n=-1,\\
e^{2\pi i\frac{1}{a_0\dotsm a_n}x^{(j)}y^{(j)}} & n\geq 0,\end{cases}
\]
where $x^{(j)}$ and $y^{(j)}$ are the truncations of $x$ and $y$. Then $z_j$ will eventually be constant, and we may argue as above to get a pairing.

Finally, it is not hard to see that $N^{(n)}\cong N^{(m)}$ for all $n,m$. In particular, $a_0N^\#=N^*$.
\end{rem}

\section{\texorpdfstring{The $a$-adic duality theorem}{The a-adic duality theorem}}

In general, note that $P^*=P$ and $S^*=S$. Hence, every subgroup $H\subset S$ acting on $N$ and $\O$ also acts on $N^*$ and $\O^*$. In particular, $\BQ(a,H)$ is well-defined if and only if $\BQ(a^*,H)$ is.

\begin{thm}\label{a-adic-duality}
Assume that $P\neq\varnothing$ and that $H$ is a nontrivial subgroup of $S$. Set $G=N\rtimes H$ and $G^*=N^*\rtimes H$. Then there is a Morita equivalence
\[
C_0(\O)\rtimes_\axb G \sim_M C_0(\R)\rtimes_\axb G^*,
\]
where the action on each side is the $ax+b$-action.
\end{thm}

\begin{proof}
It will improve the clarity of notation in this proof if we switch the stars; thus, we want to prove
\[
C_0(\O^*)\rtimes_\axb G^* \sim_M C_0(\R)\rtimes_\axb G.
\]
By \cite[Corollary~3.11]{Williams}, we can decompose both sides as iterated crossed products
\begin{align*}
C_0(\O^*)\rtimes_\axb G^*&\cong \bigl(C_0(\O^*)\rtimes_{\axb|_{N^*}} N^*\bigr)\rtimes_{\wilde{\axb}} H,\\
C_0(\R)\rtimes_\axb G&\cong \bigl(C_0(\R)\rtimes_{\axb|_N} N\bigr)\rtimes_{\wilde{\axb}} H,
\end{align*}
where in the first case
\begin{align*}
\wilde{\axb}_h\circ i_{C_0(\O^*)}&=i_{C_0(\O^*)}\circ \axb_h,\\
\wilde{\axb}_h\circ i_{N^*}&=i_{N^*}\circ \wilde{\axb}^2_h,
\end{align*}
where in turn $\wilde{\axb}^2$ is the action of $H$ on $C^*(N^*)$ given by
\[
\wilde{\axb}^2_h(g)(n)=g(h\inv n)
\]
for $g\in C_c(N^*)\subset C^*(N^*)$ and $n\in N^*$ (observe that the action of $H$ on $N^*$ preserves Haar measure because $N^*$ is discrete), and similarly for $C_0(\R)\rtimes G$.

Here, we use standard notation for the canonical embeddings into multiplier algebras of crossed products, and we identify all non-degenerate homomorphisms with their canonical extensions to multiplier algebras.

Our strategy is to find a Morita equivalence
\[
C_0(T/\O)\rtimes_{\lt} N\sim_M C_0(N\under T)\rtimes_{\rt} \O,
\]
where $T=\R\times \O$, that is equivariant for actions $\alpha$ and $\beta$ of $H$ on $C_0(T/\O)\rtimes_{\lt} N$ and $C_0(N\under T)\rtimes_{\rt} \O$, respectively, and then find isomorphisms
\begin{align*}
\bigl(C_0(T/\O)\rtimes_{\lt} N\bigr)\rtimes_\alpha H&\cong
\bigl(C_0(\R)\rtimes_{\axb|_N} N\bigr)\rtimes_{\wilde{\axb}} H,\\
\bigl(C_0(N\under T)\rtimes_{\rt} \O\bigr)\rtimes_\beta H&\cong
\bigl(C_0(\O^*)\rtimes_{\axb|_{N^*}} N^*\bigr)\rtimes_{\wilde{\axb}} H,
\end{align*}

\begin{step}
Recall that $N$ and $\O$ sit inside $T$ as closed subgroups. All the groups are abelian, and therefore, by ``Green's symmetric imprimitivity theorem'' (see for example \cite[Corollary~4.11]{Williams}) we get a Morita equivalence
\begin{equation}\label{Green}
C_0(T/\O)\rtimes_{\lt} N\sim_M C_0(N\under T)\rtimes_{\rt} \O
\end{equation}
via an imprimitivity bimodule $X$ that is a completion of $C_c(T)$. Here, $N$ acts on the left of $T/\O$ by $n\cdot ((t,y)\cdot \O)=(n+t,n+y)\cdot \O$ and $\O$ acts on the right of $N\under T$ by $(N\cdot(t,y))\cdot x=N\cdot (t,y+x)$, and the induced actions on $C_0$-functions are given by
\begin{align*}
\lt_n(f)(p\cdot \O)&=f\bigl(-n\cdot (p\cdot \O)\bigr),\\
\rt_x(g)(N\cdot p)&=g\bigl((N\cdot p)\cdot x\bigr),
\end{align*}
for $n\in N$, $f\in C_0(T/\O)$, $p\in T$, $x\in \O$, and $g\in C_0(N\under T)$. Moreover, $H$ acts by multiplication on $N$, hence on $\O$, and also on $\R$. Thus, $H$ acts diagonally on $T=\R\times\O$ by $h\cdot (t,x)=(ht,h\cdot x)$.

We will show that the Morita equivalence \eqref{Green} is equivariant for appropriately chosen actions of $H$. Define actions $\alpha$, $\beta$, and $\gamma$ of $H$ on $C_c(N,C_0(T/\O))$, $C_c(\O,C_0(T\under N))$, and $C_c(T)$ by
\begin{align*}
\alpha_h(f)(n)\bigl((t,y)\cdot\O\bigr)&=f(hn)\bigl((ht,h\cdot y)\cdot\O\bigr),\\
\beta_h(g)(x)\bigl(N\cdot (t,y)\bigr)&=\delta(h)g(h\cdot x)\bigl(N\cdot(ht,h\cdot y)\bigr),\\
\gamma_h(\xi)(t,y)&=\delta(h)^{\frac{1}{2}}\xi(ht,h\cdot y),
\end{align*}
where $\delta$ is the modular function for the multiplicative action of $H$ on $\O$, i.e., $\delta$ satisfies
\begin{equation}\label{modular}
\int_{\O}\psi(x)\,dx=\delta(h)\int_{\O}\psi(h\cdot x)\,dx
\end{equation}
for all $\psi\in C_c(\O)$. Note that $\alpha,\beta,\gamma$ are actions only because $H$ is abelian. We want $\alpha,\beta,\gamma$ to extend to give an action of $H$ on the $C_0(T/\O)\rtimes_{\lt} N - C_0(N\under T)\rtimes_{\rt} \O$ imprimitivity bimodule $X$, and we must check the conditions \cite[(4.41)-(4.44)]{Williams}.

First, for all $h\in H$, $f\in C_c(N,C_0(T/\O))$, $\xi\in C_c(T)$, and $(t,y)\in T$ we have
\begin{align*}
&\bigl(\alpha_h(f)\cdot\gamma_h(\xi)\bigr)(t,y)\\
&\quad=\sum_{n\in N} \alpha_h(f)(n)\bigl((t,y)\cdot\O\bigr)\gamma_h(\xi)\bigl(-n\cdot (t,y)\bigr)\\
&\quad=\sum_{n\in N} f(hn)\bigl((ht,h\cdot y)\cdot\O\bigr)\delta(h)^{\frac{1}{2}}\xi\bigl(h(t-n),h\cdot(y-n)\bigr)\\
&\quad=\sum_{n\in N} f(n)\bigl((ht,h\cdot y)\cdot\O\bigr)\delta(h)^{\frac{1}{2}}\xi\bigl(ht-n,h\cdot y-n\bigr)\\
&\quad=\delta(h)^{\frac{1}{2}}\sum_{n\in N} f(n)\bigl((ht,h\cdot y)\cdot\O\bigr)\xi\bigl(-n\cdot (ht,h\cdot y)\bigr)\\
&\quad=\delta(h)^{\frac{1}{2}} (f\cdot\xi)(ht,h\cdot y)\\
&\quad=\gamma_h(f\cdot\xi)(t,y),
\end{align*}
where the substitution $hn\mapsto n$ is made in the third equality.

Secondly, for all $h\in H$, $\xi\in C_c(T)$, $g\in C_c(\O,C_0(N\under T))$, and $(t,y)\in T$ we have
\begin{align*}
&\bigl(\gamma_h(\xi)\cdot\beta_h(g)\bigr)(t,y)\\
&\quad=\int_{\O} \gamma_h(\xi)\bigl((t,y)\cdot x\bigr)\beta_h(g)(-x)\bigl(N\cdot (t,y)\cdot x\bigr)\,dx\\
&\quad=\int_{\O} \delta(h)^{\frac{1}{2}}\xi\bigl(ht,h\cdot(y+x)\bigr)\delta(h)g(-h\cdot x)\bigl(N\cdot (ht,h\cdot(y+x))\bigr)\,dx\\
&\quad=\int_{\O} \delta(h)^{\frac{1}{2}}\xi(ht,h\cdot y+x)g(-x)\bigl(N\cdot (ht,h\cdot y+x)\bigr)\,dx\\
&\quad=\delta(h)^{\frac{1}{2}}\int_{\O} \xi\bigl((ht,h\cdot y)\cdot x\bigr)g(-x)\bigl(N\cdot (ht,h\cdot y)\cdot x\bigr)\,dx\\
&\quad=\delta(h)^{\frac{1}{2}}(\xi\cdot g)(ht,h\cdot y)\\
&\quad=\gamma_h(\xi\cdot g)(t,y),
\end{align*}
where in the third equality we shift $h\cdot x\mapsto x$ according to \eqref{modular}.

For the third equation, for all $h\in H$, $\xi,\eta\in C_c(T)$, $n\in N$, and $(t,y)\in T$ we have 
\begin{align*}
&\alpha_h\bigl({_L}\langle \xi,\eta\rangle\bigr)(n)\bigl((t,y)\cdot\O\bigr)\\
&\quad={_L}\langle\xi,\eta\rangle(hn)\bigl((ht,h\cdot y)\cdot\O\bigr)\\
&\quad=\int_{\O}\xi\bigl((ht,h\cdot y)\cdot x\bigr)\overline{\eta\bigl(-hn\cdot (ht,h\cdot y)\cdot x\bigr)}\,dx\\
&\quad=\int_{\O}\delta(h)\xi\bigl(ht,h\cdot y+h\cdot x\bigr)\overline{\eta\bigl(ht-hn,h\cdot y+h\cdot x-hn\bigr)}\,dx\\
&\quad=\int_{\O}\delta(h)^{\frac{1}{2}}\xi\bigl(ht,h\cdot (y+x)\bigr)\delta(h)^{\frac{1}{2}}\overline{\eta\bigl(h(t-n),h\cdot(y+x-n)\bigr)}\,dx\\
&\quad=\int_{\O}\gamma_h(\xi)\bigl((t,y)\cdot x\bigr)\overline{\gamma_h(\eta)\bigl(-n\cdot (t,y)\cdot x\bigr)}\,dx\\
&\quad={_L}\bigl\langle\gamma_h(\xi),\gamma_h(\eta)\bigr\rangle(n)\bigl((t,y)\cdot\O\bigr),
\end{align*}
where the substitution $x\mapsto h\cdot x$ is made in the fourth equality according to \eqref{modular}.

Finally, for all $h\in H$, $\xi,\eta\in C_c(T)$, $x\in\O$, and $(t,y)\in T$ we have
\begin{align*}
&\beta_h\bigl(\langle\xi,\eta\rangle_R\bigr)(x)\bigl(N\cdot(t,y)\bigr)\\
&\quad=\delta(h)\langle\xi,\eta\rangle_R(h\cdot x)\bigl(N\cdot(ht,h\cdot y)\bigr)\\
&\quad=\delta(h)\sum_{n\in N}\overline{\xi\bigl(-n\cdot (ht,h\cdot y)\bigr)}\eta\bigl((-n)\cdot (ht,h\cdot y)\cdot (h\cdot x)\bigr)\\
&\quad=\delta(h)\sum_{n\in N}\overline{\xi(ht-hn,h\cdot y-hn)}\eta\bigl(ht-hn,h\cdot y+h\cdot x-hn\bigr)\\
&\quad=\sum_{n\in N}\delta(h)^{\frac{1}{2}}\overline{\xi\bigl(h(t-n),h\cdot(y-n)\bigr)}\delta(h)^{\frac{1}{2}}\eta\bigl(h(t-n),h\cdot(y+x-n)\bigr)\\
&\quad=\sum_{n\in N}\overline{\gamma_h(\xi)\bigl(-n\cdot (t,y)\bigr)}\gamma_h(\eta)\bigl((-n)\cdot (t,y)\cdot x\bigr)\\
&=\bigl\langle\gamma_h(\xi),\gamma_h(\eta)\bigr\rangle_R(x)\bigl(N\cdot(t,y)\bigr),
\end{align*}
where we shift $n\mapsto hn$ in the fourth equality.

Hence, we can conclude that there is a Morita equivalence (see \cite{combes, cmw})
\begin{equation}\label{morita}
(C_0(T/\O)\rtimes_{\lt} N)\rtimes_{\alpha} H\sim_M (C_0(N\under T)\rtimes_{\rt} \O)\rtimes_{\beta} H.
\end{equation}

In our use of this equivariant Morita equivalence below, it will be convenient to see what the actions $\alpha$ and $\beta$ do to generators: for $h\in H$ we have
\begin{itemize}
\item $\alpha_h\circ i_{C_0(T/\O)}=i_{C_0(T/\O)}\circ \alpha^1_h$,
where $\alpha^1$ is the action of $H$ on $C_0(T/\O)$ given by
\[
\alpha^1_h(f)(p\cdot \O)=f(h\cdot p\cdot \O);
\]

\item $\alpha_h\circ i_N(n)=i_N(hn)$ for $n\in N$;

\item $\beta_h\circ i_{C_0(N\under T)}=i_{C_0(N\under T)}\circ \beta^1_h$,
where $\beta^1$ is the action of $H$ on $C_0(N\under T)$ given by
\[
\beta^1_h(g)(N\cdot p)=g(N\cdot (h\cdot p));
\]

\item $\beta_h\circ i_\O=i_\O\circ \beta^2_h$, where $\beta^2$ is the action of $H$ on $C^*(\O)$ given by
\begin{equation}\label{beta2}
\beta^2_h(g)(x)=\delta(h)g(h\cdot x)\quad\text{for $g\in C_c(\O)$ and $x\in \O$.}
\end{equation}
\end{itemize}
\end{step}

\begin{step}
The isomorphism $\psi\colon T/\O\to \R$ given by
\[
\psi\bigl((t,y)\cdot \O\bigr)=t\quad\text{for $(t,y)\in T$}
\]
transforms the action of $N$ on $T/\O$ to an action on $\R$: for $n\in N$ and $(t,y)\in T$ we have
\begin{align*}
n\cdot \psi\bigl((t,y)\cdot \O\bigr)
&=\psi\Bigl(n\cdot \bigl((t,y)\cdot \O\bigr)\Bigr)
\\&=\psi\Bigl(\bigl((n,n)+(t,y)\bigr)\cdot \O\Bigr)
\\&=\psi\bigl((n+t,n+y)\cdot \O\bigr)
\\&=n+t.
\end{align*}
The isomorphism $\psi$ induces an isomorphism $\psi_*\colon C_0(T/\O)\to C_0(\R)$ given by
\[
\psi_*(f)(t)=f(\psi\inv(t))\quad\text{for $f\in C_0(\R)$ and $t\in\R$.}
\]
The isomorphism $\psi_*$ transforms the action $\lt$ of $N$ on $C_0(T/\O)$ to an action $\rho$ on $C_0(\R)$:
for $n\in N$, $f\in C_0(\R)$, and $t\in\R$ we have
\begin{align*}
\rho_n(f)(t)
&=\psi_*\circ \lt_n\circ \psi_*\inv(f)(t)
\\&=\lt_n\circ \psi_*\inv(f)(\psi\inv(t))
\\&=\lt_n\circ \psi_*\inv(f)\bigl((t,0)\cdot \O\bigr)
\\&=\psi_*\inv(f)\bigl(-n\cdot (t,0)\cdot \O\bigr)
\\&=\psi_*\inv(f)\bigl((t-n,-n)\cdot \O\bigr)
\\&=f\Bigl(\psi\bigl((t-n,-n)\cdot \O\bigr)\Bigr)
\\&=f(t-n),
\end{align*}
and so
\[
\rho=\lt|_N,
\]
and we continue to denote this action by $\lt$.

By construction, $\psi_*$ is $N$-equivariant, and we have a corresponding isomorphism
\[
\psi_*\times N\colon C_0(T/\O)\rtimes_{\lt} N\iso C_0(\R)\rtimes_{\lt} N
\]
determined by
\begin{align*}
\psi_*\times N\circ i_{C_0(T/\O)}&=i_{C_0(T/\O)}\circ \psi_*,\\
\psi_*\times N\circ i_N^{C_0(T/\O)}&=i_N^{C_0(\R)}.
\end{align*}

We will compute the associated action of $H$ on $C_0(\R)\rtimes_{\lt} N$ by considering what $H$ does on the generators. First, $\psi$ transforms the action of $H$ on $T/\O$ to an action on $\R$: for $h\in H$ and $(t,y)\in T=\R\times \O$ we have
\begin{align*}
h\cdot \psi\bigl((t,y)\cdot \O\bigr)
&=\psi\bigl(h\cdot (t,y)\cdot \O\bigr)
\\&=\psi\bigl((ht,h\cdot y)\cdot \O\bigr)
\\&=ht.
\end{align*}
The isomorphism $\psi_*\times N$ transforms the action $\alpha$ of $H$ on $C_0(T/\O)\rtimes_{\lt} N$ to an action $\alpha'$ on $C_0(\R)\rtimes_{\lt} N$, and we compute this action on the generators. Due to how $\alpha$ acts on the generators from $C_0(T/\O)$, for $h\in H$ we have
\[
\alpha_h\circ i_{C_0(T/\O)}=i_{C_0(T/\O)}\circ \alpha'^1_h,
\]
where $\alpha'^1$ is the action of $H$ on $C_0(\R)$ given by
\[
\alpha'^1_h(f)(t)=f(ht)\quad\text{for $f\in C_0(\R)$ and $t\in\R$.}
\]
On the other hand, the action $\alpha'$ behaves the same way on generators from $N$ as $\alpha$ does:
\[
\alpha'_h\circ i_N(n)=i_N(hn)\quad\text{for $n\in N$.}
\]
\end{step}

\begin{step}
To complete the work on the crossed product $C_0(\R)\rtimes_{\axb} G$, we would like to know that our actions $\lt|_N$ of $N$ on $C_0(\R)$ and $\alpha'$ of $H$ on $C_0(\R)\rtimes_{\lt} N$ agree with the decomposition of the $ax+b$-action, so that the isomorphism
\[
\psi_*\times N\colon C_0(T/\O)\rtimes_{\lt} N\iso C_0(\R)\rtimes_{\lt} N
\]
would give an $H$-equivariant isomorphism
\[
C_0(T/\O)\rtimes_{\lt} N\iso C_0(\R)\rtimes_{\axb|_N} N,
\]
as desired. Unfortunately, we will need to tweak the action $\alpha'$ a little bit to make this come out right, as we will see below. However, we can verify immediately what we need for $\axb|_N$: if $n\in N$, $f\in C_0(\R)$, and $t\in\R$ then
\begin{align*}
\bigl(\axb|_N\bigr)_n(f)(t)
&=\axb_n(f)(t)
\\&=f(t-n)
\\&=\lt_n(f)(t),
\end{align*}
so that
\[
\axb|_N=\lt|_N.
\]
On the other hand, if $h\in H$ then
\[
\wilde{\axb}_h\circ i_{C_0(\R)}=i_{C_0(\R)}\circ \axb_h,
\]
where for $f\in C_0(\R)$ and $t\in\R$ we have
\[
\axb_h(f)(t)
=f(h\inv t),
\]
while
\[
\alpha'_h\circ i_{C_0(\R)}=i_{C_0(\R)}\circ \alpha'^1_h,
\]
where
\[
\alpha'^1_h(f)(t)=f(ht).
\]
Similarly,
\[
\wilde{\axb}_h\circ i_N(n)=i_N(h\inv n),
\]
while
\[
\alpha'_h\circ i_N(n)=i_N(hn).
\]
Thus, we see that
\[
\wilde{\axb}_h=\alpha'_{h\inv}.
\]
Therefore, to fix things up we only need to make the following adjustment to the action $\alpha'$: we compose with the inverse map on $H$, which, because $H$ is abelian and discrete, gives an action $\alpha''$ of $H$ on $C_0(\R)\rtimes_{\lt} N$, and now the above computations show that
\[
\alpha''=\wilde{\axb}.
\]
Since the isomorphism $\psi_*\times N$ is equivariant for the actions $\alpha$ and $\alpha'$ of $H$ on $C_0(T/\O)\rtimes_{\lt} N$ and $C_0(\R)\rtimes_{\axb|_N} N$, respectively, and since the crossed products
\[
\bigl(C_0(\R)\rtimes_{\axb|_N} N\bigr)\rtimes_{\alpha'} H
\]
and
\[
\bigl(C_0(\R)\rtimes_{\axb|_N} N\bigr)\rtimes_{\alpha''} H
\]
are isomorphic (because $\alpha''$ is gotten from $\alpha'$ by composing with an automorphism of $H$), we conclude that
\[
\bigl(C_0(T/\O)\rtimes_{\lt} N\bigr)\rtimes_\alpha H
\cong
\bigl(C_0(\R)\rtimes_{\axb|_N} N\bigr)\rtimes_{\alpha''} H,
\]
as desired.
\end{step}

\begin{step}
The isomorphism $\omega\colon N\under T\to \nshat$ of \thmref{duality} transforms the action of $\O$ on $N\under T$ to an action on $\nshat$: for $p\in T$, $x\in \O$, and $n\in N^*$ we have
\begin{align*}
\bigl(\omega(N\cdot p)\cdot x\bigr)(n)
&=\omega\bigl((N\cdot p)\cdot x\bigr)(n)
\\&=\omega\bigl(N\cdot (p\cdot x)\big)(n)
\\&=\bigl\langle p\cdot x,(n,n)\bigr\rangle_T
\\&=\bigl(p+(0,x),(n,n)\bigr\rangle_T
\\&=\bigl\langle p,(n,n)\bigr\rangle_T\bigl\langle (0,x),(n,n)\bigr\rangle_T
\\&=\omega(N\cdot p)(n)\langle x,n\rangle_\O,
\intertext{and, letting $\phi\colon N^*\hookrightarrow \O^*$ be the inclusion, with dual homomorphism $\hat\phi$, we continue the above as}
\\&=\omega(N\cdot p)(n)\langle x,\phi(n)\rangle_\O
\\&=\omega(N\cdot p)(n)\hat\phi(x)(n)
\\&=\bigl(\omega(N\cdot p)\hat\phi(x)\bigr)(n),
\end{align*}
and hence the action of $\O$ on the right of $\nshat$ is given by
\[
\chi\cdot x=\chi\hat\phi(x)\quad\text{for $\chi\in \nshat$.}
\]
The isomorphism $\omega$ induces an isomorphism $\omega_*\colon C_0(N\under T)\to C_0(\nshat)$ given by
\[
\omega_*(f)(\chi)=f(\omega\inv(\chi))\quad\text{for $f\in C_0(\nshat)$ and $\chi\in \nshat$.}
\]
The isomorphism $\omega_*$ transforms the action $\rt$ of $\O$ on $C_0(N\under T)$ to an action $\kappa$ on $C_0(\nshat)$: for $x\in\O$, $f\in C_0(\nshat)$, and $\chi\in \nshat$ we have
\begin{align*}
\kappa_x(f)(\chi)
&=\omega_*\circ \rt_x\circ \omega_*\inv(f)(\chi)
\\&=\rt_x\circ \omega_*\inv(f)(\omega\inv(\chi))
\\&=\omega_*\inv(f)\bigl(\omega\inv(\chi)\cdot x\bigr)
\\&=\omega_*\inv(f)\bigl(\omega\inv(\chi\cdot x)\bigr)
\\&=f(\chi\hat\phi(x)),
\end{align*}
and so
\[
\kappa=\rt\circ\hat\phi,
\]
as in \corref{C0 subgroup}.

By construction, $\omega_*$ is $\O$-equivariant, and we have a corresponding isomorphism
\[
\omega_*\times \O\colon C_0(N\under T)\rtimes_{\rt} \O\iso C_0(\nshat)\rtimes_\kappa \O
\]
determined by
\begin{align*}
\omega_*\times\O\circ i_{C_0(N\under T)}&=i_{C_0(\nshat)}\circ \omega_*,\\
\omega_*\times\O\circ i_\O^{C_0(N\under T)}&=i_\O^{C_0(\nshat)}.
\end{align*}

We will compute the associated action of $H$ on $C_0(\nshat)\rtimes_\kappa \O$ by considering what $H$ does on the generators. First, $\O$ transforms the action of $H$ on $N\under T$ to an action on $\nshat$: for $h\in H$, $(t,y)\in T=\R\times\O$, and $n\in N^*$ we have
\begin{align*}
\bigl(h\cdot \omega(N\cdot (t,y))\bigr)(n)
&=\omega\bigl(N\cdot h\cdot (t,y)\bigr)(n)
\\&=\omega\bigl(N\cdot (ht,h\cdot y)\bigr)(n)
\\&=\bigl\langle (ht,h\cdot y),(n,n)\bigr\rangle_T
\\&=\langle ht,n\rangle_\R\langle h\cdot y,n\rangle_\O
\\&=\langle t,hn\rangle_\R\langle y,hn\rangle_\O
\\&=\bigl\langle (t,y),(hn,hn)\bigr\rangle_\O
\\&=\omega\bigl(N\cdot (t,y)\bigr)(hn),
\end{align*}
and so the action of $H$ on $\nshat$ is given by
\[
(h\cdot \chi)(n)=\chi(hn)\quad\text{for $\chi\in \nshat$ and $n\in N^*$.}
\]
The isomorphism $\omega_*\times \O$ transforms the action $\beta$ of $H$ on $C_0(N\under T)\rtimes_{\rt} \O$ to an action $\beta'$ on $C_0(\nshat)\rtimes_\kappa \O$, and we compute this on the generators. Due to how $\beta$ acts on the generators from $C_0(N\under T)$, for $h\in H$ we have
\begin{align*}
\beta'_h\circ i_{C_0(\nshat)}
&=i_{C_0(\nshat)}\circ \beta'^1_h,
\end{align*}
where $\beta'^1$ is the action of $H$ on $C_0(\nshat)$ given by
\[
\beta'^1_h(f)(\chi)=f(h\cdot \chi)
\quad\text{for $f\in C_0(\nshat)$ and $\chi\in \nshat$.}
\]
On the other hand, the action $\beta'$ behaves the same way on generators from $\O$ as $\beta$ does:
\[
\beta'_h\circ i_\O=i_\O\circ \beta^2_h,
\]
where $\beta^2$ is the action of $H$ on $C^*(\O)$ from \eqref{beta2}.
\end{step}

\begin{step}
We can now apply \corref{C0 subgroup}, with the roles of $H$ and $G$ being played here by $N^*$ and $\O^*$, respectively. Thus, $N^*$ is a locally compact (discrete) group and the inclusion map $\phi\colon N^*\hookrightarrow \O^*$ is continuous. The action $\epsilon$ of the subgroup $N^*$ on $C_0(\O^*)$ is left translation composed with the inclusion $\phi$, and the action $\kappa$ of $\O$ on $C_0(\nshat)$ is right translation composed with the dual homomorphism $\hat\phi$. Thus, \corref{C0 subgroup} gives an isomorphism
\[
\tau\colon C_0(\O^*)\rtimes_\epsilon N^*\iso C_0(\nshat)\rtimes_{\kappa} \O
\]
such that
\[
\tau\circ i_{C_0(\O^*)}=i_\O\circ \FF_\O\inv \midtext{and} \tau\circ i_{N^*}=i_{C_0(\nshat)}\circ \FF_{N^*},
\]
where $\FF$ denotes the Fourier transform.
The action $\beta'$ of $H$ on $C_0(\nshat)\rtimes_\kappa \O$ corresponds under $\tau$ to an action $\beta''$ on $C_0(\O^*)\rtimes_\epsilon N^*$, which we compute on the generators: for $h\in H$ we have
\begin{align*}
\beta''_h\circ i_{C_0(\O^*)}
&=\tau\inv\circ\beta'_h\circ\tau\circ i_{C_0(\O^*)}
\\&=\tau\inv\circ\beta'_h\circ i_\O\circ \FF_\O\inv
\\&=\tau\inv\circ i_\O\circ \beta^2_h\circ \FF_\O\inv
\\&=i_{C_0(\O^*)}\circ \FF_\O\circ \beta^2_h\circ \FF_\O\inv,
\end{align*}
and we compute, for $g\in C_c(\O^*)\subset C_0(\O^*)$ and $x\in \O^*$, that
\begin{align*}
\FF_\O\circ \beta^2_h\circ \FF_\O\inv(g)(x)
&=\int_\O \langle y,x\rangle_\O \beta^2_h\circ \FF_\O\inv(g)(y)\,dy
\\&=\int_\O \langle y,x\rangle_\O \delta(h) \FF_\O\inv(g)(h\cdot y)\,dy
\\&=\int_\O \langle h\inv\cdot y,x\rangle_\O \FF_\O\inv(g)(y)\,dy
\\&=\int_\O \langle y,h\inv\cdot x\rangle_\O \FF_\O\inv(g)(y)\,dy
\\&=g(h\inv\cdot x),
\end{align*}
so that
$\beta''_h\circ i_{C_0(\O^*)}=i_{C_0(\O^*)}\circ \beta''^1_h$,
where $\beta''^1$ is the action of $H$ on $C_0(\O^*)$ given by
\[
\beta''^1_h(g)(x)=g(h\inv\cdot x).
\]
On the other hand, we have
\begin{align*}
\beta''_h\circ i_{N^*}
&=\tau\inv\circ\beta'_h\circ\tau\circ i_{N^*}
\\&=\tau\inv\circ\beta'_h\circ i_{C_0(\nshat)}\circ \FF_{N^*}
\\&=\tau\inv\circ i_{C_0(\nshat)}\circ \beta'^1_h\circ \FF_{N^*}
\\&=i_{N^*}\circ \FF_{N^*}\inv\circ \beta'^1_h\circ \FF_{N^*},
\end{align*}
so we see that
\[
\beta''_h\circ i_{N^*}=i_{N^*}\circ \beta''^2_h,
\]
where $\beta''^2$ is the action of $H$ on $C^*(N^*)$ determined by the following: for $g\in C_c(N^*)$ we have
\begin{align*}
\FF_{N^*}\bigl(\beta''^2_h(g)\bigr)(\chi)
&=\beta'^1_h\bigl(\FF_{N^*}(g)\bigr)(\chi)
\\&=\FF_{N^*}(g)(h\cdot \chi)
\\&=\sum_{n\in N^*} (h\cdot \chi)(n)g(n) 
\\&=\sum_{n\in N^*} \chi(hn)g(n) 
\\&=\sum_{n\in N^*} \chi(n)g(h\inv n), 
\end{align*}
so we see that
\[
\beta''^2_h(g)(n)=g(h\inv n)\quad\text{for $g\in C_c(N^*)$ and $n\in N^*$.}
\]
\end{step}

\begin{step}
To complete the work on the crossed product $C_0(\O^*)\rtimes_{\axb} G^*$, we show that our actions $\epsilon$ and $\beta''$ agree with the decomposition of the $ax+b$-action, so that the composition
\[
\tau\inv\circ (\omega_*\times \O)\colon C_0(N\under T)\rtimes_{\rt}\O \iso C^*(\O^*)\rtimes_\epsilon N^*
\]
gives an $H$-equivariant isomorphism
\[
C_0(N\under T)\rtimes_{\rt} \O\iso C_0(\O^*)\rtimes_{\axb|_{N^*}} N^*,
\]
which will finish the proof.

For the first, if $n\in N^*$, $f\in C_0(\O^*)$, and $x\in \O^*$ then
\begin{align*}
\bigl(\axb|_{N^*}\bigr)_n(f)(x)
&=\axb_n(f)(x)
\\&=f(x-n)
\\&=\lt_n(f)(x),
\end{align*}
so that
\[
\bigl(\axb|_{N^*}\bigr)_n
=\lt_n
=\lt_{\phi(n)}
=\epsilon_n.
\]
For the second, if $h\in H$ then
\[
\wilde{\axb}_h\circ i_{C_0(\O^*)}
=i_{C_0(\O^*)}\circ \axb_h
\]
and
\[
\beta''_h\circ i_{C_0(\O^*)}
=i_{C_0(\O^*)}\circ \beta''^1_h,
\]
and we see immediately from the definitions that $\axb_h=\beta''^1_h$, and for the generators from $N^*$ we have
\[
\wilde{\axb}_h\circ i_{N^*}
=i_{N^*}\circ \wilde{\axb}^2_h
\]
and
\[
\beta''_h\circ i_{N^*}
=i_{N^*}\circ \beta''^2_h,
\]
and again it follows immediately from the definitions that $\wilde{\axb}^2_h=\beta''^2_h$. Thus, we have shown that
\[
\wilde{\axb}_h=\beta''^2_h,
\]
and the proof is complete.
\qedhere
\end{step}
\end{proof}

\begin{rem}\label{zhang}
If $a$ is defined by $a_i=2$ for all $i$ and $H=\langle 2\rangle$, then \thmref{a-adic-duality} coincides with \cite[Theorem~7.5]{LarsenLi}, and if $a$ is the sequence described in \exref{all-primes}, it coincides with \cite[Theorem~6.5]{cuntzq}. A result related to \thmref{a-adic-duality} also appeared in a preprint of Li and L\"{u}ck \cite{LiLuck}, but was left out of the final version.

Moreover, recall that two separable $C^*$-algebras are Morita equivalent if and only if they are stably isomorphic. Hence, the $C^*$-algebras $C_0(\Omega)\rtimes_\axb G$ and $C_0(\mathbb{R})\rtimes_\axb G^*$ will actually be isomorphic by ``Zhang's dichotomy'': 
a separable, simple, purely infinite $C^*$-algebra is either unital or stable.
\end{rem}

\begin{rem}\label{empty-P}
Without any condition on $P$ (or by taking $H=\{1\}$) there is a Morita equivalence
\[
C_0(\O)\rtimes_{\lt} N \sim_M C_0(\R)\rtimes_{\lt} N^*.
\]
These algebras are also isomorphic, as indicated for the $2$-adic case in \cite[Section~5]{kps}. This holds since both algebras are stable, which is seen by describing the left and right hand side as a certain increasing union and as an inductive limit, respectively. In fact, the left hand side is in a natural way isomorphic to a stabilized Bunce-Deddens algebra. The above proof shows that the $H$-actions by multiplication on one side and inverse multiplication on the other are Morita equivalent, and therefore they are also ``outer conjugate'' by \cite[Proposition on p.~16]{combes}.

Moreover, it was explained to us by Jack Spielberg that it is possible to construct an explicit isomorphism between $C_0(\O)\rtimes_{\lt} N$ and $C_0(\R)\rtimes_{\lt} N^*$, but it is not clear whether the isomorphism is equivariant for the $H$-actions.
\end{rem}

\begin{rem}
Let $M$ be a subgroup of $N$ that is dense in $\O$, so that by \lemref{dense} $M$ is of the form $qN$ for some $q\geq 2$. By using \remref{dense-dual} to modify the proof of \thmref{a-adic-duality}, we see that when $H$ is a nontrivial subgroup of $S$ there is a Morita equivalence
\[
C_0(\O)\rtimes_\axb (M\rtimes H) \sim_M C_0(\R)\rtimes_\axb (M^*\rtimes H),
\]
where the action on each side is the $ax+b$-action.
\end{rem}

\begin{rem}\label{compensate}
If $A$ is any closed subgroup of $\R\times\O$ and $B=(\R\times\O)/A$, then
\[
C_0(A)\rtimes (N\rtimes H)\sim_M C_0(\what{B})\rtimes (N^*\rtimes H).
\]
In particular, $C^*(N\rtimes H)\sim_M C_0(\R\times\O^*)\rtimes (N^*\rtimes H)$. Moreover, with the notation from \eqref{NQ}, any group $M\rtimes H$, where $H\subset\langle P\cup Q\rangle$ acts on $M\subset N_Q$ such that $H\cap\langle P\rangle\neq\{1\}$ and $\overline{M}=N$, also gives rise to a UCT~Kirchberg algebra. In the duality theorem, we now have to compensate on the right hand side by using a larger group than $\R$.

For example, if $a=(\cdots,2,2,2,\cdots)$ and $M=\Z[\tfrac{1}{6}]$ we have
\[
C_0(\Q_2)\rtimes_\axb (\Z[\tfrac{1}{6}]\rtimes\langle 2\rangle) \sim_M C_0(\R\times\Q_3)\rtimes_\axb (\Z[\tfrac{1}{6}]\rtimes\langle 2\rangle)
\]
and
\[
C_0(\Q_2)\rtimes_\axb (\Z[\tfrac{1}{6}]\rtimes\langle 3\rangle) \sim_M C_0(\R\times\Q_3)\rtimes_\axb (\Z[\tfrac{1}{6}]\rtimes\langle 3\rangle).
\]
The algebras in the first display are UCT~Kirchberg algebras, while the actions in the second display are no longer locally contractive. This is the case since $P=\{2\}$ and $\langle 3\rangle\cap\langle 2\rangle=\{1\}$. Indeed, on the right hand side, the action of $\tfrac{1}{3}$ is contractive on $\R$ but expansive on $\Q_3$ and vice versa for $3$.
\end{rem}

\section{Invariants and isomorphism results}\label{isomorphisms}

Let $\mathbb{P}$ be the set of prime numbers.
A supernatural number is a function
\[
\lambda\colon\mathbb P\to\N\cup\{\infty\}
\]
such that $\sum_{p\in\mathbb P}\lambda(p)=\infty$.
Denote the set of supernatural numbers by $\mathbb{S}$.
It may sometimes be useful to consider a supernatural number as an infinite formal product
\[
\lambda=2^{\lambda(2)}3^{\lambda(3)}5^{\lambda(5)}7^{\lambda(7)}\dotsm.
\]

Let $\lambda$ and $\rho$ be two supernatural numbers associated with the sequence $a$ in the following way:
\[
\begin{split}
\lambda(p)&=\sup{\{i:p^i\text{ divides }a_0\dotsc a_j\text{ for some }j\geq 0\}}\in\N\cup\{\infty\}\\
\rho(p)&=\sup{\{i:p^i\text{ divides }a_{-1}\dotsc a_{-k}\text{ for some }k\geq 1\}}\in\N\cup\{\infty\}
\end{split}
\]

\begin{lem}\label{supernatural}
Let $a$ and $b$ be two sequences. The following hold:
\begin{itemize}
\item[(i)] $\Delta_a\cong\Delta_b$ if and only if $\lambda_a=\lambda_b$.
\item[(ii)] $N_a=N_b$ if and only if $\rho_a=\rho_b$.
\item[(iii)] $\UU_a=\UU_b$ if and only if both $\lambda_a=\lambda_b$ and $\rho_a=\rho_b$.
\end{itemize}
\end{lem}

\begin{proof}
From \cite[Theorem~25.16]{hr} we have
\[
\Delta\cong\prod_{p\in\lambda\inv(\infty)}\Z_p\times\prod_{p\in\lambda\inv(\N^{\times})}\Z/p^{\lambda(p)}\Z
\]
and hence (i) holds. It is not difficult to see that condition~(ii) and~(iii) also hold.
\end{proof}

This means that there is a one-to-one correspondence between supernatural numbers and noncyclic subgroups of $\Q$ containing $\Z$, and also between supernatural numbers and Hausdorff completions of $\Z$.

Condition~(iii) is equivalent to $a\sim b$, which means that there exists an isomorphism $\O_a\to\O_b$ that maps $1$ to $1$. More generally, the following result clarifies when $\O_a$ and $\O_b$ are isomorphic.

\begin{prop}\label{omega-isomorphism}
Let $a$ and $b$ be two sequences. Then $\O_a\cong\O_b$ if and only if there are natural numbers $p$ and $q$ such that
\[
a'=(\dotsc,a_{-2},qa_{-1},pa_0,a_1,\dotsc)\sim (\dotsc,b_{-2},pb_{-1},qb_0,b_1,\dotsc)=b'.
\]
\end{prop}

\begin{proof}
If $a'\sim b'$, then the map $1\mapsto\tfrac{p}{q}$ gives an isomorphism $\O_a\to\O_b$.

Suppose that there exists an isomorphism $\varphi\colon\O_a\to\O_b$. Then $\varphi(0)=0$, so open neighborhoods around $0\in\O_a$ map to open neighborhoods around $0\in\O_b$. In particular, $\varphi(\Delta_a)$ must be a compact open subgroup of $\O_b$, that is $\varphi(\Delta_a)=\overline{V}$ for some $V=\tfrac{p}{q}\Z\in\UU_b$ by \lemref{open-subgroups}. Hence, for all $U=\frac{m}{n}\Z\in\UU_a$ we have
\[
n\cdot\varphi(\overline{U})=m\cdot\varphi(\Delta_a)=m\cdot\overline{V}=\overline{m\tfrac{p}{q}\Z}.
\]
If $m'$ divides $m$ and $n'$ divides $n$, then $\frac{m'}{n'}\Z\in\UU_a$ as well. Thus, the injectivity of $\varphi$ implies that $m\tfrac{p}{q}\Z\in\UU_b$ and thus also that $\tfrac{pm}{qn}\Z\in\UU_b$, after applying \remref{automorphism}. Consequently,
\[
\varphi(\overline{\tfrac{m}{n}\Z})=\overline{\tfrac{pm}{qn}\Z},
\]
and $\tfrac{m}{n}\Z$ belongs to $\UU_a$ if and only if $\tfrac{pm}{qn}\Z$ belongs to $\UU_b$. Hence, $\UU_b=\{\tfrac{p}{q}U:U\in\UU_a\}$, so we conclude that
\[
\UU_a'=\{\tfrac{m}{n}U:m|p,n|q,U\in\UU_a\}=\{\tfrac{m}{n}V:m|q,n|p,V\in\UU_b\}=\UU_b'.
\qedhere
\]
\end{proof}

\begin{rem}
We can now see that the structure of $\O$ is preserved under the following operations:
\begin{itemize}
\item factoring out entries, that is, for $a_i=cd$, \\ $(\dotsc,a_{i-1},a_i,a_{i+1},\dotsc)\mapsto(\dotsc,a_{i-1},c,d,a_{i+1},\dotsc)$.
\item multiplying entries, that is, \\ $(\dotsc,a_{i-1},a_i,a_{i+1},a_{i+2}\dotsc)\mapsto(\dotsc,a_{i-1},a_ia_{i+1},a_{i+2},\dotsc)$.
\item interchanging entries, that is, \\ $(\dotsc,a_{i-1},a_i,a_{i+1},a_{i+2}\dotsc)\mapsto(\dotsc,a_{i-1},a_{i+1},a_i,a_{i+2},\dotsc)$.
\item shifting the sequence, that is, $a\mapsto b$, where $b_i=a_{i+n}$ for some $n\in\Z$.
\end{itemize}

However, the following operations on the sequence do not in general preserve the structure of $\O$:
\begin{itemize}
\item removing an entry, that is, \\ $(\dotsc,a_{i-1},a_i,a_{i+1},a_{i+2}\dotsc)\mapsto(\dotsc,a_{i-1},a_{i+1},a_{i+2},\dotsc)$.
\item adding an entry, that is, \\ $(\dotsc,a_{i-1},a_i,a_{i+1},\dotsc)\mapsto(\dotsc,a_{i-1},a_i,c,a_{i+1},\dotsc)$.
\item reflecting the sequence, that is, $a\mapsto b$, where $b_i=a_{-i+n}$ for some $n\in\Z$.
\end{itemize}
The first two operations preserve the structure of $\O$ if and only if the prime factors of the entries removed or added occur infinitely many times in the sequence. When $\O$ is self-dual, reflections will preserve the structure (see \propref{self-duality}).
\end{rem}

Note that if $\O_a\cong\O_b$, then $\lambda_a(p)+\rho_a(p)=\lambda_b(p)+\rho_b(p)$ for all primes $p$. In particular, $Q_a=\lambda_a\inv(0)\cap\rho_a\inv(0)=\lambda_b\inv(0)\cap\rho_b\inv(0)=Q_b$. That means that if $q\geq 2$, $qN_a$ is dense in $\O_a$ if and only if $qN_b$ is dense in $\O_b$ (see \lemref{dense} and \remref{automorphism}).

\begin{cor}\label{omega-iso}
We have $\O_a\cong\O_b$ if and only if there exists a $(\UU_a,\UU_b)$-continuous isomorphism $N_a\to N_b$.

Moreover, if $\O_a\cong\O_b$, then $P_a=P_b$, so $S_a=S_b$.
\end{cor}

\begin{proof}
This second statement holds since both $\lambda_a\inv(\infty)=\lambda_b\inv(\infty)$ and $\rho_a\inv(\infty)=\rho_b\inv(\infty)$ and thus
\[
S_a=\langle P_a\rangle=\langle \lambda_a\inv(\infty)\cap\rho_a\inv(\infty)\rangle
=\langle\lambda_b\inv(\infty)\cap\rho_b\inv(\infty)\rangle=\langle P_b\rangle=S_b.
\]
The first statement is a consequence of the proof of \propref{omega-isomorphism}. If $p$ and $q$ are as in that proof, then $N_b=\tfrac{p}{q}N_a$. Moreover, the isomorphism $N_a\to N_b$ given by $r\mapsto \tfrac{p}{q}r$ is continuous with respect to $(\UU_a,\UU_b)$ and therefore extends to an isomorphism $\O_a\to\O_b$ of the completions.
\end{proof}

If $\lambda$ is a supernatural number and $p$ is a prime, let $p\lambda$ denote the supernatural number given by $(p\lambda)(p)=\lambda(p)+1$ (with the convention that $\infty+1=\infty$) and $(p\lambda)(q)=\lambda(q)$ if $p\neq q$. The definition of $p\lambda$ extends to all natural numbers by prime factorization.

For a sequence $a$, we let $\lambda^*$ and $\rho^*$ be the supernatural numbers associated with $a^*$. Note that in general one has $\lambda^*=a_0\rho$, so $a\sim a^*$ if and only if $\lambda=a_0\rho$.

\begin{rem}\label{a-sharp}
With the notation of \subsecref{modified-dual}, we have $\lambda=\rho$ if and only if $a\sim a^\#$.
\end{rem}

\begin{cor}[of \propref{omega-isomorphism}]
Let $a$ and $b$ be two sequences. Then $\O_a\cong\O_b$ if and only if there are natural numbers $p$ and $q$ such that $p\lambda_a=q\lambda_b$ and $q\rho_a=p\rho_b$.
\end{cor}

\begin{prop}\label{self-duality}
The group of $a$-adic numbers $\O$ is self-dual if and only if there are natural numbers $p$ and $q$ such that $p\lambda=q\rho$.
\end{prop}

\begin{proof}
From \secref{dual-groups}, with reference to \cite[25.1]{hr}, we get $\O\cong\ohat$ if and only if $\O\cong\O^*$. Thus, \propref{omega-isomorphism} and the comment above imply that $\O\cong\ohat$ if and only if there are natural numbers $p',q'$ such that $p'\lambda=q'a_0\rho$. Clearly, this is equivalent to the existence of natural numbers $p$ and $q$ such that $p\lambda=q\rho$.
\end{proof}

We now discuss how the supernatural numbers are related to isomorphism invariants for the $a$-adic algebras, both directly and via the $a$-adic duality theorem.

\begin{rem}\label{q-isomorphism}
If $M_1$ and $M_2$ are two subgroups of $\Q$, then $M_1\cong M_2$ if and only if $M_1=rM_2$ for some positive rational $r$. Indeed, every homomorphism $M_i\to\Q$ is completely determined by its value at one point. Thus, if $M_1$ and $M_2$ are noncyclic subgroups of $\Q$, with associated supernatural numbers $\lambda_1$ and $\lambda_2$, then $M_1\cong M_2$ if and only if there are natural numbers $r_1$ and $r_2$ such that $r_1\lambda_1=r_2\lambda_2$. We write $\lambda_1\sim\lambda_2$ in this case.
\end{rem}

\begin{rem}\label{free-abelian}
Since $\Q^{\times}_+$ is a free abelian group (on the set of primes), this is also the case for all its subgroups. In particular, both $S$ and all $H\subset S$ are free abelian groups.
\end{rem}

\begin{prop}
Suppose $\O_a\cong\O_b$. Assume that $H$ is a subgroup of $S_a=S_b$. Then $\BQ(a,H)\cong\BQ(b,H)$.
\end{prop}

\begin{proof}
Since $\O_a\cong\O_b$, there exists an isomorphism $\omega\colon\O_a\to\O_b$ restricting to an isomorphism $N_a\to N_b$. Then the map $\varphi\colon C_c(N_a\rtimes H,C_0(\O_a))\to C_c(N_b\rtimes H,C_0(\O_b))$ given by
\[
\varphi(f)(n,h)(x)=f(\O\inv(n),h)(\O\inv(x))
\]
determines the isomorphism $\BQ(a,H)\cong\BQ(b,H)$.
\end{proof}

For two pairs of supernatural numbers $(\lambda_1,\rho_1)$ and $(\lambda_2,\rho_2)$, we write $(\lambda_1,\rho_1)\sim (\lambda_2,\rho_2)$ if there exist natural numbers $p$ and $q$ such that $p\lambda_1=q\lambda_2$ and $q\rho_1=p\rho_2$. Then the set of isomorphism classes of $a$-adic numbers coincides with $\mathbb{S}\times\mathbb{S}/\sim$, and the self-dual ones coincide with the diagonal, i.e., are of the form $[(\lambda,\lambda)]$. Moreover, the pair $([(\lambda,\rho)],H)$, where $H$ is a nontrivial subgroup of $S=\langle \lambda\inv(\infty)\cap\rho\inv(\infty)\rangle$, is an isomorphism invariant for the $a$-adic algebra $\BQ(a,H)$.

\begin{lem}
For all $H\subset S$ and rational numbers $r$ we have
\[
C_0(\R)\rtimes (N\rtimes H)\cong C_0(\R)\rtimes (rN\rtimes H).
\]
\end{lem}

\begin{proof}
The isomorphism is determined by the map
\[
\varphi\colon C_c(rN\rtimes H,C_0(\R))\to C_c(N\rtimes H,C_0(\R))
\]
given by
\[
\varphi(f)(n,h)(x)=f(rn,h)(rx).
\qedhere
\]
\end{proof}

\begin{rem}
Let $a$ be a sequence, and $N$ the associated $a$-adic rationals. Let $M$ be a subgroup of $N$ that is dense in $\O$. Let $N^{(m)}$ be the group of rationals corresponding to $a^{(m)}$. Let $a'$ be another sequence, with associated group $\O'$ and associated rationals $N'$. Suppose $\O\cong\O'$. Then
\[
\begin{split}
C_0(\R)\rtimes (N^*\rtimes H) &\cong C_0(\R)\rtimes (M^*\rtimes H)\\
C_0(\R)\rtimes (N^*\rtimes H) &\cong C_0(\R)\rtimes (N^{(m)}\rtimes H)\\
C_0(\R)\rtimes (N^*\rtimes H) &\cong C_0(\R)\rtimes ((N')^*\rtimes H).
\end{split}
\]
\end{rem}

Hence, $([\lambda],H)$, where $H$ is a nontrivial subgroup of $\langle \lambda\inv(\infty)\cap\rho\inv(\infty)\rangle$, is an isomorphism invariant for the right hand side of the $a$-adic duality theorem (\thmref{a-adic-duality}). Moreover, by \thmref{a-adic-duality} and \remref{zhang}, it should be clear that $\BQ(a,H)\cong\BQ(b,K)$ if $N_a^*\cong N_b^*$ and $H=K$, although the isomorphism is in general not canonical. Therefore, $([\lambda],H)$ is an isomorphism invariant also for $\BQ(a,H)$.

\begin{ex}
Let $a$ and $b$ be the sequences of Examples~\ref{three-at-zero} and~\ref{three-at-negative}, and let $H=\langle 2\rangle$. Then $\BQ(a,H)\cong\BQ(b,H)$ and these algebras are also isomorphic to $\BQ_2$, but this isomorphisms are not canonical.
\end{ex}

\begin{q}\label{k-theory}
Given two sequences $a$ and $b$ and subgroups $H\subset S_a$ and $K\subset S_b$. When is $\BQ(a,H)\not\cong\BQ(b,K)$? To enlighten the question, consider the following situation. Let $a=(\cdots,n,n,n,\dotsc)$ and $H=\langle n\rangle$, and note that $H=S$ if and only if $n$ is prime (see \exref{singly-generated} above). Then $\BQ(a,H)$ turns out to be isomorphic to the stabilization $\OO(E_{n,1})$ of \cite[Example~A.6]{Katsura}. 
Thus,
\[
(K_0(\QQ(a,H)),[1],K_1(\QQ(a,H)))\cong(\Z\oplus\Z/(n-1)\Z,(0,1),\Z).
\]
Moreover, since all $\BQ(a,H)$ are Kirchberg algebras in the UCT class, they are classifiable by $K$-theory.

In future work we hope to be able to compute the $K$-theory of $\BQ(a,H)$ using the following strategy: Since $C_0(\O)\rtimes N$ is stably isomorphic to the Bunce-Deddens algebra $C(\Delta)\rtimes\Z$, its $K$-theory is well-known, in fact
\begin{equation}\label{bunce-deddens}
(K_0(C(\Delta)\rtimes\Z),[1],K_1(C(\Delta)\rtimes\Z)))\cong(N^*,1,\Z).
\end{equation}
As $H$ is a free abelian group, we can apply the Pimsner-Voiculescu six-term exact sequence iteratively by adding the action of one generator of $H$ at a time. For this to work out, we will need to apply \thmref{a-adic-duality} and use homotopy arguments on the real dynamics to compute the action of $H$ on the $K$-groups (see also \cite[Remark~3.16]{CLintegral2}).
\end{q}

\begin{rem}
It is possible to compute the topological $K$-theory (in terms of complex vector bundles) of the solenoids by applying \eqref{bunce-deddens}:
\[
\begin{split}
K_i(C(\Delta)\rtimes\Z) & \cong K_i(C_0(\O)\rtimes N) \overset{\text{Rem.\ref{empty-P}}}{\cong} K_i(C_0(\R)\rtimes N^*)\\ 
&\overset{\text{Cor.\ref{C0 subgroup}}}{\cong} K_i(C(\nshat)\rtimes\what{\R})\overset{\text{Thom}}{\cong} K_{1-i}(C(\nshat))\cong K_{\textup{top}}^{1-i}(\nshat).
\end{split}
\]
That is, if $N$ is any subgroup of $\Q$, then
\[
K_{\textup{top}}^i(\nhat)=\begin{cases}\Z &\text{if } i=0 \\ N &\text{if } i=1.\end{cases}
\]
\end{rem}

\subsection{\texorpdfstring{The ring structure of $\O$}{The ring structure of Omega}}

\begin{thm}[Herman, see {\cite[12.3.35]{Palmer2}}]\label{herman-thm} The $a$-adic numbers $\O$ can be given the structure of a topological commutative ring with multiplication inherited from $N\subset\Q$ if and only if
\begin{equation}\label{herman}
N=\bigcup_{h\in S}h\Z\quad \left(=\Z\left[\left\{\tfrac{1}{p}:p\in P\right\}\right]\right).
\end{equation}
\end{thm}

Again, by the $a$-adic duality theorem (\thmref{a-adic-duality}) and Zhang's dichotomy (\remref{zhang}):

\begin{cor}
For every sequence $a$, there is a sequence $b$ such that $\O_b$ is a ring and $\BQ(a,H)\cong\BQ(b,H)$.
\end{cor}

Note that $\BQ(b,H)$ is still not a ring algebra in the sense of \cite{Li-Ring}.

\begin{proof}
Given $a$, we can find another sequence $b$ such that $S_a=S_b$, $N_a^*=N_b^*$ and $\O_b$ is a ring. More precisely, $b$ can be constructed by setting $\lambda_b=\lambda_a$ and defining $\rho_b$ by $\rho_b(p)=\infty$ if $p\in S_a$ and $\rho_b(p)=0$ if $p\notin S_a$. Then $\O_b$ is a ring by \thmref{herman-thm}. By \thmref{a-adic-duality} and \remref{zhang} (see also the comment above) we have $\BQ(a,H)\cong\BQ(b,H)$.
\end{proof}

\begin{lem}
If $\O_a$ and $\O_b$ are both rings, then they are isomorphic as topological rings if and only if $a\sim b$.
\end{lem}

\begin{rem}
Suppose $\O$ is a ring. Then $\O$ is an integral domain if and only if there is a prime $p$ such that $a_i$ is a power of $p$ for all $i$. In this case $\O$ is actually the field $\Q_p$.

Finally, $\O$ is both a ring and self-dual precisely when $\lambda(p)=\rho(p)=0$ or $\infty$ for all primes $p$. In this case, $\O$ is completely determined by the set of primes $P$.
\end{rem}

\begin{rem}\label{inverses}
Set $\UU_P=\{\tfrac{m}{n}\Z\in\UU : n\in S\}=\{U\in\UU:U\subset\Z[\{\tfrac{1}{p}:p\in P\}]\}$. Then the open subgroup
\[
R=\overline{\Z[\{\tfrac{1}{p}:p\in P\}]}=\overline{\bigcup_{U\in\UU_P}U} 
\]
in $\O$ is the maximal open (and closed) ring contained in $\O$. Indeed, by \thmref{herman-thm}, $R$ is a ring contained in $\O$. Moreover, every open ring in $\O$ must be of the form described in \lemref{open-subgroups}. We recall from \remref{automorphism} that multiplication with $\tfrac{1}{p}\in N$ in an open subgroup of $\O$ is well-defined only if it contains $\Z[\tfrac{1}{p}]$ and continuous only if $p\in P$. Hence, every open ring in $\O$ must be a subring of $R$.

Furthermore, every automorphism of $\O$ is completely determined by its value at $1$. Hence, $\aut{(\O)}$ is the subgroup of the multiplicative group $R^{\times}$ of $R$ consisting of all $x\in R^{\times}$ that has a ``unique inverse'' in $\O$. To illustrate what this means, consider \exref{three-at-negative} again and let 
\[
x=(\dotsc,0,1,1,0,1,0,1,\dotsc),
\]
where the first nonzero entry is $x_0$, and let $y$ be given by $y_{-1}=1$ and $y_i=0$ else. In this case $R=\overline{\Z[\tfrac{1}{2}]}$ and $3\in R^{\times}$, $x\in R^{\times}$ and $3x=1$. However, since $3y=1$ as well, $3\notin\aut{(\O)}$. Following the notation of \remref{automorphism}, $\{\pm r:r\in \langle P\cup Q\rangle\}$ is a subgroup of $\aut{(\O)}$.
\end{rem}

\begin{appendix}

\section{\texorpdfstring{Cuntz-Li's ``subgroup of dual group theorem''}{Cuntz-Li's "subgroup of dual group theorem"}}\label{appendix}

Our aim in this appendix is to show that \cite[Lemma~4.3]{CLintegral2} is a special case of the following result about coactions, which is probably folklore. First observe that if $\phi\colon H\to G$ is a continuous homomorphism of locally compact groups, then $\lt\circ\phi\colon H\to \aut C_0(G)$ is an action of $H$ on $C_0(G)$ and $\overline{\id\otimes\pi_\phi}\circ\delta_H$ is a coaction of $G$ on $C^*(H)$, where
\[
\pi_\phi\colon C^*(H)\to M(C^*(G))
\]
is the integrated form of $\phi$. The only property of coactions that is perhaps not obvious is injectivity, but this follows by computing that
\[
\overline{\id\otimes\pi_{1_G}}\circ\overline{\id\otimes\pi_\phi}\circ\delta_H
=\overline{\id\otimes \pi_{1_H}}\circ\delta_H=\id_{C^*(H)},
\]
where $1_G$ denotes the trivial character of $G$ and $\pi_{1_G}\colon C^*(G)\to \C$ denotes the integrated form (and similarly for $\pi_{1_H}$).

\begin{thm}\label{abstract subgroup}
Let $\phi\colon H\to G$ be a continuous homomorphism of locally compact groups, and let
\[
\epsilon=\lt\circ\phi\midtext{and}\delta=\overline{\id\otimes\pi_\phi}\circ\delta_H
\]
be the associated action of $H$ on $C_0(G)$ and coaction of $G$ on $C^*(H)$, respectively. Then there is an isomorphism
\[
\theta\colon C_0(G)\rtimes_\epsilon H \iso C^*(H)\rtimes_\delta G
\]
such that
\[
\overline\theta\circ i_{C_0(G)}=j_G
\midtext{and}
\overline\theta\circ i_H=j_{C^*(H)}.
\]
\end{thm}

\begin{proof}
It suffices to show that, given nondegenerate homomorphisms
\[
\mu\colon C_0(G)\to M(D)
\midtext{and}
\pi\colon C^*(H)\to M(D),
\]
the pair $(\mu,\pi)$ is covariant for the action $(C_0(G),H,\epsilon)$ if and only if the pair $(\pi,\mu)$ is covariant for the coaction $(C^*(H),G,\delta)$. First assume that $(\mu,\pi)$ is covariant, i.e.,
\[
\mu\circ \epsilon_t=\ad \overline \pi_t\circ\mu\quad\text{for $t\in H$.}
\]
We must show that for $t\in H$ we have
\[
\ad \overline{\mu\otimes\id}(w_G)(\overline\pi_t\otimes 1)=\overline{\pi\otimes\id}\circ\overline\delta(t),
\]
where $w_G\in C_b^\beta(G,M(C^*(G)))=M(C_0(G)\otimes C^*(G))$ denotes the canonical embedding of $G$ into $UM(C^*(G))$ (and ``$C_b^\beta$'' denotes norm-bounded functions that are strictly continuous),  or equivalently
\[
\overline{\mu\otimes\id}(w_G)(\overline\pi_t\otimes 1)=\overline{\pi\otimes\id}\circ\overline\delta(t)\overline{\mu\otimes\id}(w_G).
\]
Since the slice maps $\overline{\id\otimes h}$ for $h$ in the Fourier-Stieltjes algebra $B(G)$ separate points in $M(D\otimes C^*(G))$, it suffices to compute that
\begin{align*}
\overline{\id\otimes h}\bigl(\overline{\mu\otimes\id}(w_G)(\overline\pi_t\otimes 1)\bigr)
&=\overline{\id\otimes h}\circ\overline{\mu\otimes\id}(w_G)\overline\pi_t
\\&=\mu(h)\overline\pi_t,
\end{align*}
while
\begin{align*}
\overline{\id\otimes h}\bigl(\overline{\pi\otimes\id}\circ\overline\delta(t)\overline{\mu\otimes\id}(w_G)\bigr)
&=\overline{\id\otimes h}\bigl((\overline\pi_t\otimes \phi(t))\overline{\mu\otimes\id}(w_G)\bigr)
\\&=\overline\pi_t\,\overline{\id\otimes h}\bigl((1\otimes \phi(t))\overline{\mu\otimes\id}(w_G)\bigr)
\\&=\overline\pi_t\,\overline{\id\otimes h}\Bigl(\overline{\mu\otimes\id}\bigl((1\otimes \phi(t))w_G\bigr)\Bigr)
\\&=\overline\pi_t\,\overline{\id\otimes h}\Bigl(\overline{\mu\otimes\id}\bigl(\overline{\lt_{\phi(t)\inv}\otimes\id}(w_G)\bigr)\Bigr)
\\&=\overline\pi_t\,\overline{\id\otimes h}\Bigl(\overline{\mu\otimes\id}\bigl(\overline{\epsilon_{t\inv}\otimes\id}(w_G)\bigr)\Bigr)
\\&=\overline\pi_t\,\mu\circ\epsilon_{t\inv}\bigl(\overline{\id\otimes h}(w_G)\bigr)
\\&=\overline\pi_t\ad \overline\pi_{t\inv}\circ\mu(h)
\\&=\mu(h)\overline\pi_t.
\end{align*}

Conversely, assuming that $(\pi,\mu)$ is a covariant homomorphism of the coaction $(C^*(H),G,\delta)$, we can use much of the above computation, but now with $\omega$ in the Fourier algebra $A(G)$, getting (after replacing $t$ by $t\inv$)
\[
\mu\circ\epsilon_t(\omega)=\ad \overline\pi(t)\circ\mu(\omega),
\]
which implies that $(\mu,\pi)$ is covariant for the action $(C_0(G),H,\epsilon)$ because $A(G)$ is dense in $C_0(G)$.
\end{proof}

Now we want to make the connection with Cuntz-Li's ``subgroup of dual group theorem'' \cite[Lemma~4.3]{CLintegral2}.
So, suppose $G$ is abelian. We will want to work with the dual group of $G$, and to make the closest connection with \cite{CLintegral2} it will be better, for \corref{abelian subgroup} only, to switch the roles of $G$ and $\ghat$: so after this switch we have a continuous homomorphism $\phi\colon H\to \ghat$. Actually, for the Cuntz-Li theorem $\phi$ will be injective, and in \cite{CLintegral2} $H$ is identified with its image in $\ghat$.

\begin{rem}
There is a small difference between \corref{abelian subgroup} and \cite[Lemma~4.3]{CLintegral2}: we require $H$ to have a stronger topology than it inherits from $\ghat$, while \cite{CLintegral2} only requires the topology to make it a locally compact group such that the above $\mu$ and $\nu$ are continuous actions of $G$ and $H$ on the $C^*$-algebras $C^*(H)$ and $C^*(G)$, respectively. As Cuntz and Li mention in \cite{CLintegral2}, they are interested in the case where the topology on $H$ is discrete, so our formulation of the result is sufficient for their purposes.
\end{rem}

To prepare for the formulation of the ``subgroup of dual group theorem'', we briefly recall a bit of the theory of noncommutative duality from \cite[Appendix~A]{enchilada}. By \cite[Example~A.23]{enchilada} there is a bijective correspondence between coactions of $\ghat$ and actions of $G$: given a coaction $\delta$ of $\ghat$ on a $C^*$-algebra $A$,
the associated action $\mu$ of $G$ is given by
\[
\mu_x(a)=\overline{\id\otimes \FF^*(e_x)}\circ\delta,
\]
where $e_x\in C_0(G)^*$ is evaluation at $x$ and $\FF=\FF_{\ghat}\colon C^*(\ghat)\to C_0(G)$ is the Fourier transform, with dual map $\FF^*$. Warning: in \cite[Example~A.23]{enchilada} the convention for the Fourier transform is that $\overline\FF(\chi)(x)=\chi(x)$ for $\chi\in \ghat$ and $x\in G$; consequently, $\FF^*(e_x)$ coincides with the function in the Fourier-Stieltjes algebra $B(\ghat)=C^*(\ghat)^*$ given by the character $x$ of $\ghat$. As explained in \cite[Section~A.5]{enchilada}, the covariant representations of the coaction $(A,\ghat,\delta)$ and the action $(A,G,\mu)$ are the same modulo the isomorphism $\FF_G\colon C^*(G)\to C_0(\ghat)$, so there is an isomorphism
\[
\Upsilon\colon A\rtimes_\delta \ghat\iso A\rtimes_\mu G
\]
such that
\[
\Upsilon\circ j_A=i_A\midtext{and}\Upsilon\circ j_{\ghat}=i_G\circ\FF_{\ghat}\inv.
\]

\begin{cor}[{\cite[Lemma~4.3]{CLintegral2}}]\label{abelian subgroup}
Let $G$ be a locally compact abelian group, and let $H$ be a subgroup of the dual group $\ghat$. Let $H$ have a topology, stronger than the one it inherits from $\ghat$, that makes it a locally compact group. Then there are actions $\nu$ of $H$ on $C^*(G)$ and $\mu$ of $G$ on $C^*(H)$, given by
\begin{equation}\label{H action}
\nu_t(g)(x)=t(x)g(x)\quad\text{for $t\in H$, $g\in C_c(G)\subset C^*(G)$, and $x\in G$,}
\end{equation}
and
\begin{equation}\label{G action}
\mu_x(f)(t)=\overline{t(x)}f(t)\quad\text{for $x\in G$, $f\in C_c(H)\subset C^*(H)$, and $t\in H$,}
\end{equation}
respectively.

Moreover, there is an isomorphism
\[
\sigma\colon C^*(G)\rtimes_\nu H\xrightarrow{\cong} C^*(H)\rtimes_\mu G
\]
such that for $g\in C_c(G)$ and $f\in C_c(H)$ the image $\sigma(i_{C^*(G)}(g)i_H(f))$ coincides with the element of
\[
C_c(G,C_c(H))\subset C_c(G,C^*(H))
\]
given by
\[
\sigma\bigl(i_{C^*(G)}(g)i_H(f)\bigr)(x)(t)=\overline{t(x)}g(x)f(t).
\]
\end{cor}

\begin{proof}
This will follow quickly from \thmref{abstract subgroup} and the above facts relating coactions and actions, modulo one extra step: we will need to compose with the inverse in both $H$ and $G$ to get the actions in the precise form of the statement of the corollary. We do this in order to get as close as possible to Cuntz-Li's ``subgroup of dual group theorem'', and this extra adjustment is necessitated by our nonstandard convention for the Fourier transform.

The hypotheses tell us that, in the notation of \thmref{abstract subgroup} (but again with $G$ replaced by $\ghat$), the homomorphism $\phi\colon H\to \ghat$ is the inclusion map, so the action $\epsilon$ of $H$ on $C_0(\ghat)$ is just the restriction of $\lt$ to $H$. The Fourier transform
\[
\FF_G\inv\colon C_0(\ghat)\iso C^*(G)
\]
transforms the action $\epsilon$ to an action $\tilde\nu$ of $H$ on $C^*(G)$. We compute that for $t\in H$ the automorphism $\tilde\nu_t$ of $C^*(G)$ is the integrated form of the homomorphism $V_t\colon G\to M(C^*(G))$ given by
\[
V_t(x)
=\FF_G\inv\circ\epsilon_t\circ\overline\FF_G(x)
=\FF_G\inv\Bigl(\overline{t(x)}\FF_G(x)\Bigr)
=\overline{t(x)}x.
\]
Thus, for $g\in C_c(G)\subset C^*(G)$ we have
\[
\tilde \nu_t(g)=\int_G g(x)V_t(x)\,dx
=\int_G g(x)\overline{t(x)}x\,dx,
\]
so
\[
\tilde \nu_t(g)(x)=\overline{t(x)}g(x).
\]
On the other hand, we can let $\tilde\mu$ be the action of $G$ on $C^*(H)$ corresponding to the coaction $\delta$ of $\ghat$. Then for $x\in G$ the automorphism $\tilde\mu_x$ of $C^*(H)$ is the integrated form of the homomorphism $U_x\colon H\to M(C^*(H))$ given by
\begin{align*}
U_x(t)
&=\overline{\id\otimes \FF_{\ghat}^*(e_x)}\circ\overline\delta(t)
=\overline{\id\otimes \FF_{\ghat}^*(e_x)}(t\otimes t)
\\&=\FF_{\ghat}^*(e_x)(t)t
=t(x)t.
\end{align*}
Thus, for $f\in C_c(H)\subset C^*(H)$ we have
\[
\tilde\mu_x(f)(t)=t(x)f(t).
\]
We now compose with the inverse in both $H$ and $G$ to get actions $\nu$ of $H$ on $C^*(G)$ and $\mu$ of $G$ on $C^*(H)$ as in \eqref{H action} and \eqref{G action}, respectively.

Finally, for $g\in C_c(G)$ and $f\in C_c(H)$ we have
\begin{align*}
\sigma\bigl(i_{C^*(G)}(g)i_H(f)\bigr)
&=i_G(g)i_{C^*(H)}(f)
\\&=\int_G g(x)i_G(x)i_{C^*(H)}(f)\,dx
\\&=\int_G g(x)i_{C^*(H)}(\mu_x(f))i_G(x)\,dx
\\&=\int_G i_{C^*(H)}\bigl(g(x)\mu_x(f)\bigr)i_G(x)\,dx,
\end{align*}
so $\sigma\bigl(i_{C^*(G)}(g)i_H(f)\bigr)$ coincides with the element of $C_c(G,C_c(H))$ given by
\[
\sigma\bigl(i_{C^*(G)}(g)i_H(f)\bigr)(x)
=g(x)\mu_x(f),
\]
and evaluating this function at $t\in H$ gives
\[
\sigma\bigl(i_{C^*(G)}(g)i_H(f)\bigr)(x)(t)
=g(x)\mu_x(f)(t)
=\overline{t(x)}g(x)f(t).
\qedhere
\]
\end{proof}

Now we will present a third version of the ``subgroup of dual group theorem''. In contrast to the Cuntz-Li version, which involves actions on the group $C^*$-algebras of $G$ and $H$, for our purposes it will be more convenient to have a version of \thmref{abstract subgroup} with actions on the $C_0$-functions on both sides. Also, we will now switch the roles of $G$ and $\ghat$ back, and we will not require $H$ to embed injectively into $G$:

\begin{cor}\label{C0 subgroup}
Let $\phi\colon H\to G$ be a continuous homomorphism of locally compact abelian groups, with dual homomorphism
\[
\hat\phi\colon\ghat\to\hhat.
\]
Let
\[
\epsilon=\lt\circ\phi\midtext{and}\kappa=\rt\circ\hat\phi
\]
be the associated actions of $H$ on $C_0(G)$ and $\ghat$ on $C_0(\hhat)$, respectively. Then there is an isomorphism
\begin{equation}\label{tau}
\tau\colon C_0(G)\rtimes_\epsilon H \iso C_0(\hhat)\rtimes_\kappa \ghat
\end{equation}
such that
\[
\overline\tau\circ i_{C_0(G)}=i_{\ghat}\circ \FF_{\ghat}\inv
\midtext{and}
\overline\tau\circ i_H=i_{C_0(\hhat)}\circ \FF_H.
\]
\end{cor}

\begin{proof}
From \thmref {abstract subgroup} we have a coaction $\delta$ of $G$ on $C^*(H)$ and an isomorphism
\begin{equation}\label{theta}
\theta\colon C_0(G)\rtimes_\epsilon H\iso C^*(H)\rtimes_\delta G
\end{equation}
such that
\[
\theta\circ i_{C_0(G)}=j_G\midtext{and}\theta\circ i_H=j_{C^*(H)}.
\]
As we explained above \corref{abelian subgroup}, the coaction $\delta$ of $G$ corresponds to an action of $G$ on $C^*(H)$. We used this in \corref{abelian subgroup}, but there were a couple of differences between the contexts there and here, so to avoid confusion we do the computation anew, with fresh notation: we denote the associated action by $\kappa^0$, and compute that for $\chi\in \ghat$ the automorphism $\kappa^0_\chi$ of $C^*(H)$ is the integrated form of the homomorphism $R_\chi\colon H\to M(C^*(H))$ given by
\begin{align*}
R_\chi(t)
&=\bigl(\id\otimes\FF_G^*(e_\chi)\bigr)\circ\delta(t)
\\&=\bigl(\id\otimes\FF_G^*(e_\chi)\bigr)\bigl(t\otimes \phi(t)\bigr)
\\&=\chi(\phi(t))t,
\end{align*}
so that for $f\in C_c(H)\subset C^*(H)$ we have
\[
\kappa^0_\chi(f)=(\chi\circ\phi)f.
\]
Now, due to our convention regarding the Fourier transform, for $\zeta\in \hhat$ and $f\in C_c(H)$ we have
\[
\FF_H(\zeta f)=\rt_\zeta\circ\FF_H(f).
\]
Thus, the isomorphism $\FF_H\colon C^*(H)\to C_0(\hhat)$ carries the action $\kappa^0$ to an action $\kappa^1$ of $G$ on $C_0(\hhat)$ given by
\begin{align*}
\kappa^1_\chi(\FF_H(f))
&=\FF_H(\kappa^0_\chi(f))
\\&=\FF_H\bigl((\chi\circ\phi)f\bigr)
\\&=\rt_{\chi\circ\phi}(\FF_H(f)),
\end{align*}
so that $\kappa^1$ is given on $g\in C_0(\hhat)$ by
\[
\kappa^1_\chi(g)=\rt_{\hat\phi(\chi)}(g)
\]
and hence $\kappa^1$ agrees with the action $\kappa=\rt\circ\hat\phi$ defined in the statement of the corollary.

Now we trace the effects of the various transformations as we convert the isomorphism $\theta$ of \eqref{theta} to the isomorphism $\tau$ of \eqref{tau}: we have
\begin{align*}
\tau\circ i_{C_0(G)}
&=(\FF_H\times \ghat)\circ \Upsilon\circ \theta\circ i_{C_0(G)}
\\&=(\FF_H\times \ghat)\circ \Upsilon\circ j_G
\\&=(\FF_H\times \ghat)\circ i_{\ghat}^{C^*(H)}\circ \FF_{\ghat}\inv
\\&=i_{\ghat}^{C_0(\hhat)}\circ \FF_{\ghat}\inv
\end{align*}
and
\begin{align*}
\tau\circ i_H
&=(\FF_H\times \ghat)\circ \Upsilon\circ \theta\circ i_H
\\&=(\FF_H\times \ghat)\circ \Upsilon\circ j_{C^*(H)}
\\&=(\FF_H\times \ghat)\circ i_{C^*(H)}
\\&=i_{C_0(\hhat)}\circ \FF_H
\qedhere
\end{align*}
\end{proof}

\end{appendix}

\bibliographystyle{plain}

\end{document}